\documentclass[12pt,twosides]{amsart}
\usepackage{amssymb,amsmath,amsthm, amscd, enumerate, mathrsfs}
\usepackage{graphicx, hhline}
\usepackage[all]{xy}
\usepackage{braket}
\usepackage[usenames]{color}
\usepackage{hyperref}
\usepackage{fancyhdr}
\usepackage{bm}
\usepackage{url}
\usepackage[top=30truemm,bottom=30truemm,left=30truemm,right=30truemm]{geometry}

\hypersetup{colorlinks=true}

\title{Semistable reduction for complex analytic spaces}
\author{Makoto Enokizono, Kenta Hashizume}
\address{Graduate School of Mathematical Sciences, University of Tokyo,
3-8-1 Komaba, Meguro-ku, Tokyo 153-8914, Japan }\email{enokizono@g.ecc.u-tokyo.ac.jp}
\address{Department of Mathematics, Faculty of Science, Niigata University, Niigata 950-2181, Japan}\address{Institute for Research Administration, Niigata University, Niigata 950-2181, Japan}\email{hkenta@math.sc.niigata-u.ac.jp}

\newtheorem{thm}{Theorem}[section]

\newtheorem{lem}[thm]{Lemma}
\newtheorem{cor}[thm]{Corollary}
\newtheorem{prop}[thm]{Proposition}

\newtheorem{ques}[thm]{Question}

\theoremstyle{definition}
\newtheorem{defn}[thm]{Definition}
\newtheorem{rem}[thm]{Remark}

\newtheorem{exam}[thm]{Example}
\newtheorem*{ack}{Acknowledgments} 
 
\newtheorem*{b-divisor}{b-divisors} 
 
\newtheorem*{g-pair}{Generalized pairs} 
\newtheorem*{adj-g-pair}{Divisorial adjunction for generalized pairs} 
\newtheorem*{mmp-g-pair}{MMP for generalized pairs}

\newtheorem*{claim*}{Claim}

\def\C{\mathbb C}
\def\R{\mathbb R}
\def\Q{\mathbb Q}
\def\Z{\mathbb Z}

\def\O{\mathcal{O}}

\DeclareMathOperator{\Hom}{Hom}

\begin{document}

\maketitle

\begin{abstract}
Semistable reduction theorem for projective morphisms in the category of complex analytic spaces is established.
\end{abstract}

\tableofcontents

\section{Introduction}

Hironaka's resolution theorem stands as a cornerstone in both algebraic and analytic geometry.
Semistable reduction theorem, which guarantees 
the existence of relative resolutions for morphisms between varieties, is also crucial.
In essence, the semistable reduction problem seeks to find an alteration $B'\to B$ and a modification $X'\to X\times_{B}B'$ from smooth varieties for a given surjective morphism $f\colon X\to B$ between varieties such that the induced morphism $f'\colon X'\to B'$ has fibers with mild singularities.
For the case where $\dim B=1$, the semistable reduction theorem was initially established in \cite{KKMS} for both algebraic and analytic varieties.
For higher dimensional bases $B$, Abramovich and Karu \cite{AbKa} proved a weakly semistable reduction theorem within the category of algebraic varieties in characteristic $0$.
They also showed the full semistable reduction theorem can be deduced from a purely combinatorial conjecture known as the polyhedral conjecture. 
Recently, Adiprasito, Liu and Temkin \cite{ALT} proved the polyhedral conjecture and hence the semistable reduction theorem holds in the category of algebraic varieties in characteristic $0$.
In this paper, we prove the semistable reduction theorem in the category of complex analytic varieties:

\begin{thm}[Simplified version of Theorem~\ref{semistred}, see also Remark~\ref{shrink}] \label{Introsemistred}
Let $f\colon X\to B$ be a projective surjective morphism between complex analytic varieties.
Assume that $B$ is projective over a Stein space and fix any  compact subset $K$ of $B$.
Then after shrinking $B$ to an open neighbourhood of $K$, there exist a projective alterations $B'\to B$ from a smooth variety $B'$, a projective modification $X'\to X\times_{B}B'$ from a smooth variety $X'$ onto the main component of $X\times_{B}B'$ such that the projection $f'\colon X'\to B'$ is semistable.
\end{thm}

Here a morphism $f\colon X\to B$ between smooth varieties is said to be {\em semistable} (with respect to $\Delta_{X}$ and $\Delta_{B}$) if there exist reduced divisors $\Delta_{X}\subseteq X$ and $\Delta_{B}\subseteq B$ with simple normal crossing supports such that $f$ can locally be written as
$$
f(x_{1},\ldots,x_{n})=(t_{1},\ldots, t_{m}),\quad t_{i}=\prod_{j=l_{i-1}+1}^{l_{i}}x_{j}
$$
for some $0=l_{0}<l_{1}<\cdots<l_{m}\le n$,
where $(x_{1},\ldots,x_{n})$ and $(t_{1},\ldots, t_{m})$ are respectively local coordinates on $X$ and $B$ such that $\Delta_{X}=\{x_{1}\cdots x_{l}=0\}$ and $\Delta_{B}=\{t_{1}\cdots t_{k}=0\}$ for some $k$ and $l\ge l_{k}$ satisfying $l_{i+1}=l_{i}+1$ for $i\ge k$.

The proof of Theorem~\ref{Introsemistred} heavily relies on the arguments in \cite{AbKa} and employs toroidal geometry.
Building upon the results of \cite{ALT}, the semistable reduction problem is reduced to the toroidal reduction problem for morphisms.
By the argument of \cite{AbKa}, it suffices to establish the semistable reduction theorem for morphisms of relative dimension one, which is the complex analytic counterpart of de Jong's results \cite{deJ}.
However, there are some obstacles to discussing in the complex analytic category.
Indeed, de Jong obtained the semistable reduction for families of curves $f\colon X\to B$ by resolving the indeterminacy of a certain moduli map $B'\dasharrow M$, where $M$ is a fine moduli space of curves with additional structure and $B'\to B$ an alteration such that the base change of $f$ to some open subspace of $B'$ admits the moduli map.
However, in the complex analytic category, there does not always exist a resolution of indeterminacy (for instance, the exponential map $\exp\colon \mathbb{C}\to \mathbb{C}\subseteq \mathbb{P}^1$ is not extended to $\mathbb{P}^1\to \mathbb{P}^1$).
To avoid this difficulty, we propose an alternate approach.
Let $f\colon X\to B$ be a projective surjective morphism between complex analytic varieties
and assume for simplicity that all irreducible components of all fibers of $f$ are curves of geometric genus at least $2$.
We first take a resolution of singularities $B'\to B$, ensuring that the discriminant locus of the base change of $f$ to $B'$ is simple normal crossings.
Next we directly construct a ``semistable reduction in codimension one''.
Consequently, we may assume that $f\colon X\to B$ defines a family of stable curves over some Zariski open subspace $U\subseteq B$ the complement of which has codimension at least $2$ and the discriminant locus of $f$ is simple normal crossings.
Finally, we apply the extension theorem of de Jong--Oort \cite{dJOo} to obtain a family of stable curves $f'\colon X'\to B$, which is a desired semistable reduction.
Notably, our proof circumvents the use of the moduli spaces of curves and morphisms. 
After the paper was completed, Dan Abramovich taught the authors a reduction of Theorem~\ref{Introsemistred} to the algebraic case. 
For details, see Subsection \ref{subsecreduction}.

The shrinking assumption and the projectivity assumption in Theorem~\ref{Introsemistred} are necessary for technical reasons.
Hence the following question naturally arises:

\begin{ques}
Can the semistable reduction theorem (Theorem~\ref{semistred}) and the toroidal reduction theorem (Theorem~\ref{toroidalred}) be established without assuming the shrinking  or the projectivity condition, while allowing for base changes by infinite discrete coverings?
\end{ques}

This question leads a more natural question:
\begin{ques}
Can the semistable reduction theorem (Theorem~\ref{semistred}) and the toroidal reduction theorem (Theorem~\ref{toroidalred}) be extended functorially?
\end{ques}
In the category of algebraic varieties (or more broadly, Noetherian schemes) in characteristic $0$,
the functorial semistable and toroidal reduction problem have been developed (e.g., \cite{Mol}, \cite{ATW}). 
Other developments in a wide category containing complex analytic varieties are explored in \cite{ATW} and \cite{BidS}.

The contents of this paper are as follows: 
Section \ref{sec2} provides some notions on toroidal geometry and finite coverings of complex analytic spaces used in this paper. 
In Section \ref{sec3}, we prove the semistable reduction theorem for families of curves in the complex analytic setting.
In Section \ref{sec4}, we prove the main results in this paper. 
We also give a reduction of Theorem~\ref{Introsemistred} to the algebraic case. 

\begin{ack} The first author was partially supported by JSPS KAKENHI Grant Number JP20K14297. The second author was partially supported by JSPS KAKENHI Grant Number JP22K13887. The authors are grateful to Professor Osamu Fujino for useful advice and informing them of the paper \cite{Kar}. The authors are also grateful to Professor Dan Abramovich for informing them of the reduction of the main theorem to the algebraic case. 
The authors thank Professor Christopher D. Hacon for a comment. 
The authors thank the referee for a suggestion. 
\end{ack}

\section{Preliminaries} \label{sec2}
In this paper, {\em complex analytic spaces} are always assumed to be Hausdorff and second countable.
{\em Analytic varieties} are reduced and irreducible complex analytic spaces. 
For the notations and definitions, see \cite{Fuj} (see also \cite{EH}). 

\begin{defn}[Projective morphism]
Let $f\colon X\to Y$ be a proper morphism of complex analytic spaces and $\mathcal{L}$ an invertible $\O_{X}$-module on $X$.
Then $\mathcal{L}$ on $X$ is said to be {\em relatively very ample over $Y$} if the natural map $f^{*}f_{*}\mathcal{L}\to \mathcal{L}$ is surjective and the induced morphism $X\to \mathbb{P}_{Y}(f_{*}\mathcal{L})$ is a closed immersion.
The invertible sheaf $\mathcal{L}$ on $X$ is called {\em relatively ample over $Y$} if for any point $y\in Y$, there exists an open neighbourhood $U$ of $y$ and a positive integer $m$ such that the restriction $\mathcal{L}^{\otimes m}|_{f^{-1}(U)}$ is relatively very ample over $U$.
The morphism $f$ is said to be {\em projective} if there exists a relatively ample sheaf on $X$ over $Y$.
\end{defn}

\begin{lem} \label{projemb}
Let $f\colon X\to B$ be a projective morphism of complex analytic spaces and assume that $B$ is projective over a Stein space $S$.
Let $W$ be any Stein compact subset of $S$.
Then after shrinking $S$ around $W$, there is a closed immersion $X \hookrightarrow \mathbb{P}^{N}\times B$ over $B$ for some $N$.
\end{lem}

\begin{proof}
Let $\mathcal{L}$ be a relatively ample sheaf on $X$ over $B$.
Since $W$ is compact and $B$ is projective over $S$, there exists a positive integer $m$ such that $\mathcal{L}^{\otimes m}$ is relatively very ample over $B$ after shrinking $S$ around $W$.
Let $\mathcal{M}$ be a relatively ample sheaf on $B$ over $S$.
By the compactness of $W$ and {\cite[Chapter~IV, \S 2, Theorem~2.1]{BaSt}}, there exists a positive integer $n$ such that the natural map 
$$
\pi^{*}\pi_{*}(f_{*}\mathcal{L}^{\otimes m}\otimes \mathcal{M}^{\otimes n})\to f_{*}\mathcal{L}^{\otimes m}\otimes \mathcal{M}^{\otimes n}
$$ 
is surjective along $\pi^{-1}(W)$, where $\pi\colon B\to S$ is the structure morphism.
Since $S$ is Stein, the coherent sheaf $\pi_{*}(f_{*}\mathcal{L}^{\otimes m}\otimes \mathcal{M}^{\otimes n})$ is globally generated by Cartan's Theorem~A.
Hence $f_{*}\mathcal{L}^{\otimes m}\otimes \mathcal{M}^{\otimes n}$ is globally generated after shrinking $S$ around $W$.
Since $\pi^{-1}(W)$ is compact and $f_{*}\mathcal{L}^{\otimes m}\otimes \mathcal{M}^{\otimes n}$ is coherent, there exists a finite dimensional vector subspace $V$ of $H^{0}(B, f_{*}\mathcal{L}^{\otimes m}\otimes \mathcal{M}^{\otimes n})$ such that 
$$
V\otimes \O_{B}\to f_{*}\mathcal{L}^{\otimes m}\otimes \mathcal{M}^{\otimes n}
$$
is surjective after shrinking $S$ around $W$.
Thus the composition 
$$
X\hookrightarrow \mathbb{P}_{B}(f_{*}\mathcal{L}^{\otimes m})=\mathbb{P}_{B}(f_{*}\mathcal{L}^{\otimes m}\otimes \mathcal{M}^{\otimes n})\hookrightarrow \mathbb{P}(V)\times B
$$
is a desired closed immersion.
\end{proof}

\begin{rem}[Shrinking assumption] \label{shrink}
In this paper, we consider the situation that the base space $B$ is projective over a Stein space $S$ and fix any Stein compact subset $W$ of $S$.
Almost all results in this paper are of the form that after shrinking $S$ around $W$, the given statement holds.
The arguments of this form lead to the arguments of the form that after shrinking $B$ around any compact subset $K$ of $B$, the given statement holds.
Indeed, let $\pi\colon B\to S$ be the structure morphism.
Then $\pi(K)$ is also compact.
We take the holomorphically convex hull $W$ of $\pi(K)$.
Then $W$ is a Stein compact subset of $S$ from {\cite[Lemma~2.5]{Fuj}} and $K\subseteq \pi^{-1}(W)$.
\end{rem}

\subsection{Toroidal geometry}
In this subsection, we recall the notion of toroidal embeddings introduced in \cite{KKMS} with reference to {\cite[Section~1]{AbKa}}.

\begin{defn}[Toroidal pair]
Let $X$ be a normal analytic variety and $U\subseteq X$ a Zariski open subset.
Then the inclusion $U\subseteq X$ is called a {\em toroidal embedding}
if for any point $x\in X$, there exist a germ of a toric variety $t\in X_{\sigma}$ 
and an isomorphism $\hat{\O}_{X, x}\cong \hat{\O}_{X_{\sigma}, t}$ of complete local $\C$-algebras which sends the reduced ideal sheaf of $X\setminus U$ in $\hat{\O}_{X, x}$ to that of $X_{\sigma}\setminus T_{\sigma}$ in $\hat{\O}_{X_{\sigma}, t}$,
where we write $X_{\sigma}$ as an affine toric variety associated to a strongly convex rational polyhedral cone $\sigma$ with the open dense torus $T_{\sigma}$.
We call the germ $t\in X_{\sigma}$ a {\em local model at $x$}.
Replacing $X_{\sigma}$ by a neighbourhood of $t$ if necessary, we may assume that the orbit of $t$ is the unique closed orbit in $X_{\sigma}$ and hence we can replace $t$ with the distinguished point $x_{\sigma}$.
Let $\Delta$ be the complement of $U$ in $X$.
Then $\Delta$ is a reduced divisor since $X_{\sigma}\setminus T_{\sigma}$ is the sum of torus-invariant divisors for any local model $t\in X_{\sigma}$.
The pair $(X, \Delta)$ is called a {\em toroidal pair}.
A toroidal pair $(X, \Delta)$ is {\em strict} (or {\em without self-intersection}) if each irreducible component of $\Delta$ is normal.
\end{defn}

\begin{rem}
Note that the boundary divisor $\Delta$ may have infinite irreducible components in general.
However, from the next section onwards, we will only deal with the case of finitely many irreducible components since we shrink the base Stein space $S$ around any Stein compact subset $W$.
\end{rem}

\begin{defn}[Strata]
Let $(X, \Delta)$ be a strict toroidal pair and $\Delta=\sum_{i\in I}E_{i}$ the irreducible decomposition. 
Then all the connected components $Y_{\lambda}$ of $\cap_{j\in J}E_{j}\setminus \cup_{i\in I\setminus J}E_{i}$ for each subset $J\subseteq I$ form a stratification $X=\sqcup_{\lambda}Y_{\lambda}$.
For a stratum $Y=Y_{\lambda}$, we define
$\mathrm{Star}(Y)$ as the union of strata $Z$ such that the closure of $Z$ contains $Y$, which is Zariski open in $X$ and contains $Y$ as an analytic subset.
Let $M^{Y}$ denote the group of Cartier divisors on $\mathrm{Star}(Y)$ whose supports are contained in $\Delta$, and put $M^{Y}_{\R}:=M^{Y}\otimes_{\Z} \R$, $N^{Y}:=\Hom(M^{Y},\Z)$ and $N^{Y}_{\R}:=N^{Y}\otimes_{\Z} \R$.
Let $M^{Y}_{+}$ denote the submonoid of $M^{Y}$ generated by effective divisors and put
$$
\sigma^{Y}:=(M^{Y}_{+})^{\vee}=\{\gamma\in N^{Y}_{\R}\ |\ \langle D, \gamma\rangle \ge 0 \text{ for any $D\in M^{Y}_{+}$}\}.
$$
According to {\cite[Corollary~1, p.61]{KKMS}}, for any local model $x_{\sigma}\in X_{\sigma}$ at a point $x\in Y$ where $\sigma$ is the cone in $N_{\sigma,\R}$, there are natural identifications
$$
M^{Y}\cong M_{\sigma}/\sigma^{\bot}\cap M_{\sigma},\quad N^{Y}\cong N_{\sigma}\cap \mathrm{Span}(\sigma), \quad \sigma^{Y}\cong \sigma.
$$
\end{defn}

\begin{defn}[Toroidal morphism] \label{tormordef}
Let $(X,\Delta_{X})$ and $(Y, \Delta_{Y})$ be toroidal pairs.
Then a dominant morphism $f\colon X\to Y$ is {\em toroidal} if
for any point $x\in X$, there exist a local model $s\in X_{\sigma}$ at $x$, a local model $t\in X_{\tau}$ at $f(x)$ and a toric morphism $g\colon X_{\sigma}\to X_{\tau}$ which maps $s$ to $t$ such that the formal isomorphisms $(x\in X)\cong (s\in X_{\sigma})$ and $(f(x)\in Y)\cong (t\in X_{\tau})$ are compatible with $f$ and $g$.
\end{defn}


The following results are useful. 

\begin{prop}[{\cite[Proposition~1.5]{AbKa}}]
Let $f\colon X\to Y$ be a toroidal morphism and $x$ a point of $X$.
Then for any local model $x_{\tau}\in X_{\tau}$ at $f(x)$, there exist a local model $x_{\sigma}\in X_{\sigma}$ at $x$ and a toric morphism $g\colon X_{\sigma}\to X_{\tau}$ 
which maps $x_{\sigma}$ to $x_{\tau}$ such that the formal isomorphisms $(x\in X)\cong (x_{\sigma}\in X_{\sigma})$ and $(f(x)\in Y)\cong (x_{\tau}\in X_{\tau})$ are compatible with $f$ and $g$.
\end{prop}

\begin{cor}[{\cite[Corollary~1.6]{AbKa}}]
Let $f\colon X\to Y$ and $g\colon Y\to Z$ be toroidal morphisms.
Then $g\circ f\colon X\to Z$ is also toroidal.
\end{cor}

\begin{defn}[Rational conical polyhedral complex]
A {\em rational conical polyhedral complex}  is the data $\Sigma=(|\Sigma|, \{(\sigma, M_{\sigma})\})$ as follows:

\begin{enumerate}
\item
$|\Sigma|$ is a topological space.

\item
$\{(\sigma, M_{\sigma})\}$ is a set of pairs $(\sigma, M_{\sigma})$ such that $\sigma$ is a subset of $|\Sigma|$, which is called a {\em cone in $\Sigma$}, and $M_{\sigma}$ is a finitely generated subgroup of the vector space of continuous $\R$-valued functions on $\sigma$, which is called the {\em lattice of functions on $\sigma$}, satisfying as follows:
\begin{enumerate}
\item
For any $\Z$-basis $f_{1},\ldots,f_{n_{\sigma}}$ of $M_{\sigma}$,
the continuous map 
$$
\varphi_{\sigma}:=(f_{1}\ldots,f_{n_{\sigma}})\colon \sigma\to \R^{n_{\sigma}}
$$
is a homeomorphism onto the image $\sigma':=\varphi_{\sigma}(\sigma)$ which is a full-dimensional strongly convex rational polyhedral cone in $\R^{n_{\sigma}}$.
Moreover, for any face $\tau'\preceq \sigma'$, there exists a pair $(\tau, M_{\tau})$ such that $\tau=\varphi_{\sigma}^{-1}(\tau')$.
In this case, we denote by $\tau\preceq \sigma$ and $\tau$ is called a {\em face of $\sigma$}.

\item
For any $\tau\preceq \sigma$, the group $M_{\tau}$ coincides with the set of restricted functions $f|_{\tau}\colon \tau\to \R$ for $f\in M_{\sigma}$.
In particular, there is a natural surjective homomorphism $M_{\sigma}\to M_{\tau}$.

\item
$|\Sigma|$ is the union of cones $\sigma$ and the intersection $\sigma\cap \sigma'$ for any cones $\sigma$ and $\sigma'$ is a finite union of faces of both $\sigma$ and $\sigma'$.
\end{enumerate}

\end{enumerate}
Put $N_{\sigma}:=\Hom(M_{\sigma},\Z)$ and call it the {\em lattice of $\sigma$}.
Note that the cone $\sigma$ naturally embeds into $N_{\sigma,\R}:=N_{\sigma}\otimes \R$ by the composite map $\sigma\cong \sigma'\subseteq \R^{n_{\sigma}}\cong N_{\sigma,\R}$, where the last map is determined by the basis $f_{1},\ldots, f_{n_{\sigma}}$.
This embedding $\sigma\subseteq N_{\sigma,\R}$ does not depend on the choice of the basis $f_{1},\ldots, f_{n_{\sigma}}$.
The condition (b) is equivalent to $N_{\tau}=N_{\sigma}\cap \mathrm{Span}(\tau)\subseteq N_{\sigma}$.
Hence we can describe $\Sigma$ as the pairs $(\sigma, N_{\sigma})$ with $\sigma\subseteq N_{\sigma, \R}$ satisfying certain conditions instead of the pairs $(\sigma, M_{\sigma})$.
\end{defn}

\begin{exam}
Let $X_{\Sigma}$ be a toric variety associated to a fan $\Sigma$ in $N_{\R}$.
Then the associated fan $\Sigma$ can be considered as a rational conical polyhedral complex as
$$
|\Sigma|=\bigcup_{\sigma\in \Sigma}\sigma,\quad N_{\sigma}=N\cap \mathrm{Span}(\sigma),\quad M_{\sigma}=\Hom(N_{\sigma},\Z)\cong M/\sigma^{\bot}\cap M.
$$
\end{exam}

\begin{thm}[{\cite[p.71]{KKMS}}]
Let $(X,\Delta)$ be a strict toroidal pair.
Then the data 
$$
\Sigma_{(X,\Delta)}:=(|\Sigma_{(X,\Delta)}|, \{(\sigma^{Y},M^{Y})\}_{Y})
$$
 is a rational conical polyhedral complex, where $Y$ runs through all the strata of $(X, \Delta)$,
the restriction map $M^{Y}\to M^{Z};\ D\mapsto D|_{\mathrm{Star}(Z)}$ is surjective for any strata $Y$ and $Z$ with $Y\subseteq \overline{Z}$ and
$|\Sigma_{(X,\Delta)}|$ is the gluing of $\sigma^{Y}$ for all strata $Y$.
\end{thm}

\begin{defn}[Morphism of polyhedral complexes]
Let $\Sigma=(|\Sigma|, \{(\sigma, M_{\sigma})\})$ and $\Sigma'=(|\Sigma'|, \{(\sigma', M_{\sigma'})\})$ be rational conical polyhedral complexes.
We define a {\em morphism of rational conical polyhedral complexes} $\varphi\colon \Sigma\to \Sigma'$ as a continuous map $\varphi\colon |\Sigma|\to |\Sigma'|$ satisfying that for any cone $\sigma$ in $\Sigma$, there exists a cone $\sigma'$ in $\Sigma'$ such that $\varphi(\sigma)\subseteq \sigma'$ and the composite function $f\circ \varphi|_{\sigma}\colon \sigma\to \R$ belongs to $M_{\sigma}$ for each $f\in M_{\sigma'}$.
Note that in this case, the dual $\varphi_{\sigma}\colon N_{\sigma}\to N_{\sigma'}$ of the induced homomorphism $M_{\sigma'}\to M_{\sigma}$ satisfies 
$\varphi_{\sigma, \R}|_{\sigma}=\varphi|_{\sigma}$ as a map from $\sigma$ to $\sigma'$.
If, moreover, $|\Sigma'|=|\Sigma|$ holds and the continuous map $\varphi$ and all homomorphisms between lattices are identity maps, then $\Sigma'$ is called a {\em subdivision of $\Sigma$}.

For a toroidal morphism $f\colon X\to Y$ between strict toroidal pairs $(X,\Delta_{X})$ and $(Y, \Delta_{Y})$, the induced morphism $\varphi_{f}\colon \Sigma_{(X,\Delta_{X})}\to \Sigma_{(Y,\Delta_{Y})}$ of the associated polyhedral complexes can be defined as follows (cf.\ {\cite[Lemma~1.8]{AbKa}}):
Let $x\in X$ be a point, $x_{\sigma}\in X_{\sigma}$ and $x_{\tau}\in X_{\tau}$ local models at $x$ and $f(x)$, and $g\colon X_{\sigma}\to X_{\tau}$ a toric morphism compatible with their local models.
Let $Z\subseteq X$ and $W\subseteq Y$ denote the strata containing $x$ and $f(x)$, respectively.
Since the toric morphism $g$ sends the orbit of $x_{\sigma}$ into the orbit of $x_{\tau}$, we have $f(Z)\subseteq W$ and hence $f(\mathrm{Star}(Z))\subseteq \mathrm{Star}(W)$.
Then we can define the homomorphism $f^{*}\colon M^{W}\to M^{Z}$ of lattices by the pullback of Cartier divisors, which is compatible with $g^{*}\colon M_{\tau}\to M_{\sigma}$.
It induces the maps $f^{*}\colon M_{+}^{W}\to M_{+}^{Z}$ and $\varphi^{Z}=(f^{*})^{\vee}\colon \sigma^{Z}\to \sigma^{W}$.
These maps $\{\varphi^{Z}\}_{Z}$ are glued together to define the continuous map $\varphi_{f}\colon |\Sigma_{(X,\Delta_{X})}|\to |\Sigma_{(Y,\Delta_{Y})}|$,
which defines a morphism of their polyhedral complexes.
\end{defn}

\begin{lem} \label{subdivmodif}
Let $(X, \Delta_{X})$ be a strict toroidal pair.
Then there is a one-to-one correspondence between the isomorphism classes of proper toroidal morphisms $f\colon X'\to X$ from a toroidal pair $(X', \Delta_{X'})$ with $X'\setminus \Delta_{X'}\cong X\setminus \Delta_{X}$
and subdivisions $\Sigma'$ of the associated polyhedral complex $\Sigma_{(X, \Delta_{X})}$.
\end{lem}

\begin{proof}
Let $f\colon X'\to X$ be a proper toroidal morphism from a toroidal pair $(X', \Delta_{X'})$ satisfying $X'\setminus \Delta_{X'}\cong X\setminus \Delta_{X}$.
Note that $(X', \Delta_{X'})$ is automatically strict since so is $(X,\Delta_{X})$ and each exceptional prime divisor of $f$ is normal which is due to the toric case. 
Let $\varphi_{f}\colon \Sigma_{(X', \Delta_{X'})}\to \Sigma_{(X, \Delta_{X})}$ denote the associated morphism of polyhedral complexes.
For any cone $\sigma'\in \Sigma_{(X', \Delta_{X'})}$ and $\sigma \in \Sigma_{(X, \Delta_{X})}$ with $\varphi_{f}(\sigma')\subseteq \sigma$, the corresponding toric morphism $g\colon X_{\sigma'}\to X_{\sigma}$ between local models is \'{e}tale over the torus of $X_{\sigma}$ by the assumption $X'\setminus \Delta_{X'}\cong X\setminus \Delta_{X}$.
It follows that $N_{\sigma',\R}=N_{\sigma, \R}$ and hence $\varphi_{f}$ is injective.
Since $f$ is proper, $\varphi_{f}$ is surjective. 
Thus $\Sigma':=\Sigma_{(X', \Delta_{X'})}$ is a subdivision of $\Sigma_{(X, \Delta_{X})}$.

Let $\Sigma'$ be a subdivision of $\Sigma_{(X, \Delta_{X})}$.
For each cone $\sigma'\in \Sigma'$ and $\sigma^{Y}\in \Sigma_{(X, \Delta_{X})}$ with $\sigma'\subseteq \sigma^{Y}$, we define an analytic variety
$$
X_{\sigma', Y}:=\mathrm{Specan}_{\mathrm{Star}(Y)}\left(\sum_{D\in \sigma'^{\vee}\cap M^{Y}}\O_{\mathrm{Star}(Y)}(-D) \right).
$$
Let $f_{\sigma', Y}\colon X_{\sigma', Y}\to \mathrm{Star}(Y)$ be the natural projection and put $\Delta_{\sigma', Y}:=f_{\sigma', Y}^{-1}(\Delta_{X})$.
One can easily check that these data $\{f_{\sigma', Y}\colon X_{\sigma', Y}\to \mathrm{Star}(Y), \Delta_{\sigma', Y}\}_{\sigma',\sigma^{Y}}$ are glued together to a proper toroidal morphism $f\colon X'\to X$ from a strict toroidal variety $(X',\Delta_{X'})$ with $\Sigma_{(X',\Delta_{X'})}=\Sigma'$.
\end{proof}

\begin{lem} [{\cite[Lemma~1.11]{AbKa}}] \label{subdivtoroidal}
Let $f\colon X\to B$ be a toroidal morphism between strict toroidal pairs and $\varphi_{f}\colon \Sigma_{X}\to \Sigma_{B}$ the associated morphism of polyhedral complexes.
Let $X'\to X$ and $B'\to B$ be toroidal modifications associated to subdivisions $\Sigma_{X'}$ and $\Sigma_{B'}$ of $\Sigma_{X}$ and $\Sigma_{B}$, respectively.
Then there exists a toroidal morphism $f'\colon X'\to B'$ which is a lifting of $f$ if and only if for each $\sigma \in \Sigma_{X'}$, there exists a cone $\tau\in \Sigma_{B'}$ such that $\varphi_{f}(\sigma)\subseteq \tau$.
\end{lem}

\begin{defn}[Resolution, projective subdivision and alteration]
Let $\Sigma$ be a rational conical polyhedral complex.
We say that $\Sigma$ is {\em non-singular} (resp.\ {\em simplicial}) if each cone $\sigma\in \Sigma$ is generated by some $\Z$-basis of $N_{\sigma}$ (resp.\ $\Q$-basis of $N_{\sigma,\Q}$).
Let $\varphi\colon \Sigma'\to \Sigma$ be a subdivision of $\Sigma$.
Then $\varphi$ is called a {\em resolution} of $\Sigma$ if $\Sigma'$ is non-singular.

The subdivision $\varphi$ is said to be {\em projective} if there exists a continuous function $\psi\colon |\Sigma|\to \R$, which is called a {\em good function associated to $\varphi$}, such that for each cone $\sigma\in \Sigma$, the restriction $\psi|_{\sigma}$ is convex and piecewise linear, takes rational values at all lattice points and the largest pieces in $\sigma$ where $\psi$ is linear are cones in $\Sigma'$.
By {\cite[Chapter~III, Theorem~4.1]{KKMS}}, there always exists a projective resolution of $\Sigma$.

Let $\varphi\colon \Sigma'\to \Sigma$ be a morphism of rational conical polyhedral complexes.
Then $\varphi$ is said to be a {\em lattice alteration} if the continuous map $\varphi$ is a homeomorphism and via this identification $|\Sigma'|=|\Sigma|$, the polyhedral complex $\Sigma'$ can be written by the data $\{(\sigma, N'_{\sigma})\}_{\sigma\in \Sigma}$, where $N'_{\sigma}$ is a sublattice of $N_{\sigma}$ for each cone $\sigma$ in $\Sigma$.
The morphism $\varphi\colon \Sigma'\to \Sigma$ is called an {\em alteration} if $\varphi$ is a composition 
$$
\Sigma'\xrightarrow{\varphi'} \Sigma_{1}\xrightarrow{\varphi_{1}} \Sigma,
$$
where $\varphi'$ is a lattice alteration and $\varphi_{1}$ is a subdivision.
If, moreover, $\varphi_{1}$ is projective, we say that $\varphi$ is a {\em projective alteration}.
\end{defn}

\begin{rem}
Note that given two morphisms $\varphi_{i}\colon \Sigma_{i}\to \Sigma$ ($i=1,2$) such that $\varphi_{1}$ is an alteration, the fiber product $\Sigma_{1}\times_{\Sigma}\Sigma_{2}$ can be written as follows:
The topological space $|\Sigma_{1}\times_{\Sigma}\Sigma_{2}|$ equals $|\Sigma_{2}|$, the cones and the lattices are respectively of the form $\sigma_{1}\times_{\sigma}\sigma_{2}$ and $N_{\sigma_{1}}\times_{N_{\sigma}}N_{\sigma_{2}}$ where $\sigma_{i}\in \Sigma_{i}$, $\sigma\in \Sigma$ with $\varphi_{i}(\sigma_{i})\subseteq \sigma$.
\end{rem}

\begin{lem}  
Let $f\colon X'\to X$ be a toroidal modification between strict toroidal pairs associated to a subdivision $\varphi_{f}\colon \Sigma_{X'}\to \Sigma_{X}$.
Then $f$ is projective if and only if $\varphi_{f}$ is projective.
\end{lem}

\begin{proof}
Let $\Delta_{X'}=\sum_{i=1}^{n}E_{i}$ be the irreducible decomposition of the boundary divisor of $X'$.
Then each prime divisor $E_{i}$ corresponds to a $1$-dimensional cone $\rho_{i}\in \Sigma_{X'}$.
Let $u_{\rho_{i}}\in N_{\rho_{i}}\cap \rho_{i}$ denote the primitive vector for each $\rho_{i}$.
For a Cartier divisor $D=\sum_{i=1}^{n}a_{i}E_{i}$ on $X'$ whose support is contained in $\Delta_{X'}$, we can attach the piecewise linear function $\psi_{D}\colon |\Sigma_{X'}|\to \R$ which maps $u_{\rho_{i}}$ to $-a_{i}$.
Similar to the toric case {\cite[Theorem~7.2.11]{CLS}}, $D$ is relatively ample over $X$ if and only if $\psi_{D}$ is a good function associated to the subdivision.
If $f$ is projective, then there exists a relatively ample Cartier divisor $D$ on $X'$ whose supports are contained in the boundary divisor $\Delta_{X'}$. 
Hence the claim holds.
\end{proof}

\begin{defn}[Weakly semistable and semistable]
Let $\varphi\colon \Sigma'\to \Sigma$ be a morphism of rational conical polyhedral complexes.
Then $\varphi$ is called {\em equidimensional} if for each cone $\sigma\in \Sigma'$, the image $\varphi(\sigma)$ is also a cone in $\Sigma$.
The morphism $\varphi$ is said to be {\em weakly semistable} (resp.\ {\em semistable}) if the following conditions (i), (ii) and (iii) (resp.\ (i), (ii), (iii) and (iv)) hold:

\smallskip

\noindent
(i) $\varphi$ is equidimensional.

\smallskip

\noindent
(ii) For each cone $\sigma\in \Sigma'$, the homomorphism $\varphi_{\sigma}\colon N_{\sigma}\to N_{\varphi(\sigma)}$ is surjective.

\smallskip

\noindent
(iii) $\Sigma$ is non-singular.

\smallskip

\noindent
(iv) $\Sigma'$ is non-singular.
\end{defn}

\begin{rem}
In \cite{AbKa}, the definition of (weak) semistability requires the condition that $\varphi^{-1}(0)=0$.
The above definition follows \cite{ALT}.
\end{rem}

\begin{lem} \label{toroidalmorproperty}
Let $f\colon (X, \Delta_{X})\to (B, \Delta_{B})$ be a toroidal morphism between strict toroidal pairs and $\varphi_{f}\colon \Sigma_{X}\to \Sigma_{B}$ the associated morphism of polyhedral complexes.
Then the following hold:

\smallskip

\noindent
(1) $f(\Delta_{X})\subseteq \Delta_{B}$ if and only if $\varphi_{f}^{-1}(0)=0$.

\smallskip

\noindent
(2) $f$ is equidimensional if and only if $\varphi_{f}$ is equidimensional.

\smallskip

\noindent
(3) Assume that $f$ is equidimensional and $B$ is non-singular.
Then any fiber of $f$ is reduced if and only if $\varphi_{f}$ is weakly semistable.

\smallskip

\noindent
(4) $f$ is semistable with respect to $\Delta_{X}$ and $\Delta_{B}$ if and only if $\varphi_{f}$ is semistable.
\end{lem}

\begin{proof}
The claim (1) follows from the fact that strata $Z\subseteq X$ and $W\subseteq B$ satisfy $f(Z)\subseteq W$ if and only if $\sigma^{W}$ is the minimal cone in $\Sigma_{B}$ containing $\varphi_{f}(\sigma^{Z})$.
The claim (2) is due to {\cite[Lemma~4.1]{AbKa}}.
The claim (3) is due to {\cite[Lemma~5.2]{AbKa}}.
The claim (4) follows from calculations of local equations.
\end{proof}

\begin{defn}[Toroidal action and Pre-toroidal action \cite{AbdJ}]
Let $(X, \Delta)$ be a toroidal pair.
Let $\rho\colon G \curvearrowright (X,\Delta)$ be a faithful finite group action on the pair $(X, \Delta)$, that is, a finite group $G$ acts on $X$ faithfully such that $\Delta$ is $G$-invariant.

\smallskip

\noindent
(1) $\rho$ is {\em toroidal at $x\in X$} if there exists a local model $x_{\sigma}\in X_{\sigma}$ at $x$ and a group homomorphism $\varphi\colon G_{x}:=\mathrm{Stab}_{G}(x)\to \mathrm{Stab}_{T_{\sigma}}(x_{\sigma})$ from the stabilizer subgroup of $x$ to that of $x_{\sigma}$ such that the formal isomorphism $(x\in X)\cong (x_{\sigma}\in X_{\sigma})$ is $G_{x}$-equivariant.
We say that $\rho$ is {\em toroidal} if it is toroidal at any $x\in X$.
In this case, its quotient $(X/G, \Delta/G)$ is also a toroidal pair since the quotient $X_{\sigma}/G_{x}$ is toric.

\smallskip

\noindent
(2) $\rho$ is {\em pre-toroidal at $x\in X$} if there exist a germ of a toric variety $x_{\sigma}\in X_{\sigma}$, a group homomorphism $\varphi\colon G_{x}\to \mathrm{Stab}_{T_{\sigma}}(x_{\sigma})$, a non-trivial character $\psi$ of $G_{x}$ and a $G_{x}$-equivariant formal isomorphism 
$$
(x\in X, \Delta)\cong ((x_{\sigma}, 0)\in X_{\sigma}\times \C, (X_{\sigma}\setminus T_{\sigma})\times \C),
$$
where the action of $G_{x}$ on $X_{\sigma}\times \C$ is defined by the diagonal action $(\varphi, \psi)$.
Note that the character $\psi$ depends only on the germ $x\in X$.
We say that $\rho$ is {\em pre-toroidal} if it is toroidal or pre-toroidal for any $x\in X$.

\smallskip

\noindent
(3) Assume that $\rho$ is pre-toroidal.
Then the {\em torific ideal} $\mathcal{I}_{\rho}$ of $\rho$ is defined as follows:
For any $x\in X$, let $\psi_{x}$ define the character $\psi$ of $G_{x}$ as in (2) when $\rho$ is pre-toroidal at $x$ or put $\psi_{x}:=1$ when $\rho$ is toroidal at $x$.
Since $G_{x}$ is abelian, the stalk $\O_{X, x}$ has the eigenspace decomposition
$$
\O_{X, x}=\bigoplus_{\gamma\in \Hom(G_{x}, \C^{*})}\O_{X, x}^{(\gamma)}.
$$
The ideal $\mathcal{I}_{\rho}$ is defined as the ideal generated by the subspaces $\O_{X, x}^{(\psi_{x})}$ for all $x\in X$.
Then the zero-set of the torific ideal coincides with the locus of points at which $\rho$ is pre-toroidal.
\end{defn}

\begin{lem}[{\cite[Theorem~1.7]{AbdJ}}] \label{torify}
Let $\rho\colon G \curvearrowright (X,\Delta)$ be a pre-toroidal action.
Then the normalized blowing-up $\pi \colon X'\to X$ along the torific ideal of $\rho$ defines the toroidal pair $(X', \Delta')$, where $\Delta'$ is the inverse image of the union of $\Delta$ and the locus of points at which $\rho$ is pre-toroidal.
Moreover, the action $\rho$ lifts to a toroidal action $\rho'\colon G \curvearrowright (X',\Delta')$.
\end{lem}

\subsection{Abhyankar's lemma}
In this subsection, we will prove an analytic version of Abhyankar's lemma which is used later and discuss finite coverings of analytic varieties.

\begin{lem}[Uniqueness of finite coverings] \label{uniqfin}
Let $X$ be a normal analytic variety of dimension $n$ and $Y\subseteq X$ a non-empty Zariski open subspace.
Let $\pi_{i}\colon X_{i}\to X$, $i=1,2$ be finite surjective morphisms from normal analytic varieties.
Let $\varphi\colon \pi_{1}^{-1}(Y)\cong \pi_{2}^{-1}(Y)$ be an isomorphism over $Y$.
Then it extends uniquely an isomorphism $X_{1}\cong X_{2}$ over $X$.
\end{lem}

\begin{proof}
Let $Y_{i}:=\pi_{i}^{-1}(Y)$.
We first show that the map $\varphi\colon X_{1}\dasharrow X_{2}$ is bimeromorphic.
Let $\Gamma\subseteq Y_{1}\times Y_{2}$ be the graph of $\varphi$ and show that the closure of $\Gamma$ in $X_{1}\times X_{2}$ is an analytic subset.
Note that $\Gamma$ is an irreducible component of $Y_{1}\times_{Y}Y_{2}$ of dimension $n$ and $Y_{1}\times_{Y}Y_{2}\subseteq X_{1}\times_{X}X_{2}$ is Zariski open whose complement has dimension less than $n$.
Thus applying Remmert--Stein's theorem (cf.\ {\cite[Theorem~4.6]{Shi}}), the closure $\overline{\Gamma}$ of $\Gamma$ in $X_{1}\times_{X}X_{2}$ is an analytic subset. 
Let $X_{3}$ denote the normalization of $\overline{\Gamma}$.
Then the decompositions $X_{3}\to X_{i}\to X$, $i=1,2$ are both the Stein factorizations.
Hence the uniqueness of the Stein factorization says that $X_{1}$ is isomorphic to $X_{2}$ over $X$ which is an extension of $\varphi$.
\end{proof}



The following is an analytic version of Abhyankar's lemma (cf.\ \cite{GrMu}).

\begin{prop}[Analytic version of Abhyankar's lemma] \label{analAbh}
Let $X$ be a smooth analytic variety of dimension $n$, $\Delta$ a reduced normal crossing divisor on $X$ and $Y:=X\setminus \Delta$.
Let $\pi\colon Y'\to Y$ be a finite \'{e}tale covering.
Then there exist a toroidal pair $(X', \Delta')$ with $X'\setminus \Delta'\cong Y'$ and a finite toroidal morphism $(X', \Delta')\to (X, \Delta)$ compatible with $\pi$ such that
the following hold:

\smallskip

\noindent
(i) If $\Delta$ is a smooth divisor, then $X'$ and $\Delta'$ are smooth.

\smallskip

\noindent
(ii) If $\pi$ is Galois with the Galois group $G$, then so is $X'\to X$ and the action of $G$ on $X'$ is toroidal. 
\end{prop}

\begin{proof}
For any point $p\in X$, we take an open neighbourhood $U\subseteq X$ such that $U\cong \mathbb{D}^{n}$, where $\mathbb{D}$ denotes the unit open disc in $\mathbb{C}$.
We may assume that $p$ is the origin and $\Delta\cap U$ is defined by $x_{1}x_{2}\cdots x_{k}=0$ for some $0\le k\le n$.
Let 
$$
V:=U\setminus \Delta\cong \mathbb{D}^{*k}\times \mathbb{D}^{n-k},\quad \mathbb{D}^{*}:=\mathbb{D}\setminus \{0\}
$$
and consider any connected component $V'$ of $\pi^{-1}(V)$.
The restriction map $V'\to V$ of $\pi$ is also finite \'{e}tale.
Let $\widetilde{V}$ denote the universal cover of $V$.
Then we have the decomposition $\widetilde{V}\to V'\to V$.
Note that the exponential map $\bm{e}\colon \mathbb{H}\to \mathbb{D}^{*}; z\mapsto e^{2\pi\sqrt{-1}z}$ gives the universal covering of $\mathbb{D}^{*}$ with the covering transformation group $\Z$ which naturally extends to the universal covering $\bm{e}\colon \C\to \C^{*}$.
Thus the covering $\widetilde{V}\to V$ is isomorphic to 
$$
\bm{e}^{k}\times \mathrm{id}\colon \mathbb{H}^{k}\times \mathbb{D}^{n-k}\to \mathbb{D}^{*k}\times \mathbb{D}^{n-k}
$$
with the covering transformation group $\pi_{1}(V)\cong \Z^{k}$ which naturally extends to 
$$
\bm{e}^{k}\times \mathrm{id}\colon \C^{k}\times \C^{n-k}\to \C^{*k}\times \C^{n-k}.
$$
Since $\pi_{1}(V')$ is a subgroup of the abelian group $\pi_{1}(V)\cong \Z^{k}$, the covering $V'\to V$ is Galois with the finite Galois group $\pi_{1}(V)/\pi_{1}(V')$.
Note that $\pi_{1}(V')$ is (abstractly) isomorphic to $\Z^{n}$ and so $V'\cong \widetilde{V}/\pi_{1}(V')$ is isomorphic to $\mathbb{D}^{*k}\times \mathbb{D}^{n-k}$.
Now we regard $V$ as an open subspace of the toric variety $X_{\sigma, N}=\C^{*k}\times \C^{n-k}$ with the same fundamental group,
where $N:=\Z^{n}$ is the lattice and the cone $\sigma\subseteq N_{\R}=\R^{n}$ is generated by $e_{k+1},\ldots,e_{n}$, where $\{e_{i}\}_{i=1}^{n}$ is the standard basis of $N$.
Let $N_{\sigma}$ be the sublattice of $N$ spanned by $\sigma$, which has the basis $e_{k+1},\ldots, e_{n}$.
By {\cite[Proposition~12.1.9]{CLS}}, we have a natural isomorphism $\pi_{1}(X_{\sigma, N})\cong N/N_{\sigma}$ induced by the natural identification $\pi_{1}(T_{N})\cong N$. 
The subgroup $\pi_{1}(V')\subseteq \pi_{1}(V)=\pi_{1}(X_{\sigma, N})$ corresponds to an intermediate lattice $N_{\sigma}\subseteq N'\subseteq N$ with $\pi_{1}(V')\cong N'/N_{\sigma}$.
The inclusion $N'\subseteq N$ induces the toric morphism $X_{\sigma, N'}\to X_{\sigma, N}$ which is finite \'{e}tale and corresponds to the inclusion $\pi_{1}(X_{\sigma, N'})=N'/N_{\sigma}\subseteq \pi_{1}(X_{\sigma, N})=N/N_{\sigma}$.
Thus $\widetilde{V}\to V'\to V$ extends to the sequence of covering maps
$$
\bm{e}^{k}\times \mathrm{id} \colon \C^{k}\times \C^{n-k}\to X_{\sigma, N'} \to X_{\sigma, N}.
$$

Let $\tau$ denote the cone in $N_{\R}$ generated by $e_{1},\ldots, e_{k}$ and $\sigma$.
Then the inclusion $V=U\setminus \Delta\subseteq X_{\sigma, N}$ extends to $U\subseteq X_{\tau, N}=\C^{k}\times \C^{n-k}$.
Moreover, the finite \'{e}tale morphism $X_{\sigma, N'}\to X_{\sigma, N}$ naturally extends to a finite toric morphism $\pi'\colon X_{\tau, N'}\to X_{\tau, N}$.
Note that $\tau$ is simplicial in $N'_{\R}$, and smooth if $k=0$ or $k=1$.
Let $U':=\pi'^{-1}(U)$.
Then $V'$ is a Zariski open subspace in $U'$ and the restriction map $U'\to U$ is an extension of the covering $V'\to V$.
The pair $(U', \Delta'_{U'}:=U'\setminus V')$ is toroidal and $\pi'\colon (U', \Delta'_{U'})\to (U, \Delta|_{U})$ is a toroidal morphism since it is the restriction of the toroidal embedding $X_{\sigma, N'}\subseteq X_{\tau, N'}$ and $\pi' \colon X_{\tau, N'}\to X_{\tau, N}$ is a toric morphism.
By Lemma~\ref{uniqfin}, the above $(U', \Delta'_{U'})$ constructed for each $p\in U\subseteq X$ and $V'$ can naturally glued to obtain a toroidal pair $(X', \Delta')$ with $X'\setminus \Delta'=Y'$ and a finite toroidal morphism $\pi\colon (X', \Delta')\to (X, \Delta)$ which is an extension of $\pi\colon Y'\to Y$.

If $\Delta$ is smooth, then $k=0$ or $k=1$ for the above proof and hence $X'$ and $\Delta'$ are smooth by construction, whence (i) holds.

If $\pi\colon Y'\to Y$ is Galois with the Galois group $G$, then $V'\to V$ as above is  a $G_{V'}$-covering, where $G_{V'}$ is the stabilizer group of the connected component $V'$.
Then $U'\to U$ is also a $G_{V'}$-covering and $\pi^{-1}(U)\to U$ becomes a $G$-covering (this can also be seen by Lemma~\ref{uniqfin}).
For any $x\in X'$, we put $p:=\pi(x)\in X$ and take $U$ and $V'$ as above such that $x\in U'$.
Then the stabilizer group $G_{x}$ at $x$ equals to $G_{V'}\cong \pi_{1}(V)/\pi_{1}(V')\cong N/N'$.
Hence the action of $G_{x}$ on the germ $(x\in X')\cong (x_{\tau}\in X_{\tau, N'})$ factors through the torus action $T_{N'}$ on $X_{\tau, N'}$, whence (ii) holds.
\end{proof}

As a corollary, we obtain the following Grauert--Remmert's theorem.

\begin{cor}[Extension of finite coverings {\cite[Theorem XII.5.4]{GrRa}}] \label{extfin}
Let $X$ be a normal analytic variety and $Y\subseteq X$ a non-empty Zariski open subspace.
Let $\pi\colon Y'\to Y$ be a finite covering from a normal analytic variety $Y'$.
Assume that the branch locus $B_{Y}$ of $\pi$ extends to an analytic subset $B_{X}$ of $X$.
Then there exists a finite covering $X'\to X$ from a normal analytic variety $X'$ which is an extension of $\pi$. 
\end{cor}

\begin{proof}
We take a resolution $\widetilde{X}\to X$ of $X$ such that the inverse image of the union of $X\setminus Y$ and $B_{X}$ has simple normal crossing support, which is denoted by $\Delta$.
Applying Proposition~\ref{analAbh} to $(\widetilde{X}, \Delta)$ and the normalized base change $\widetilde{Y'}\to \widetilde{Y}\subseteq \widetilde{X}$ of $Y'\to Y\subseteq X$,
we have a finite covering $\widetilde{X'}\to \widetilde{X}$ which is compatible with $\pi$.
Taking the Stein factorization $\widetilde{X'}\to X'\to X$ of the composite map $\widetilde{X'}\to \widetilde{X}\to X$, we have the desired morphism $X'\to X$.
\end{proof}

\begin{lem}[Existence of the Galois closure] \label{Galclo}
Let $\pi\colon X'\to X$ be a finite covering between normal analytic varieties.
Then there exists a finite Galois covering $\varphi\colon X''\to X$ from a normal analytic variety $X''$ which factors through $\pi$ which is universal in the following sense:
For any finite Galois covering $\psi\colon Y\to X$ from a normal analytic variety which factors through $\pi$, there exists uniquely a Galois covering $Y\to X''$ over $X'$.
\end{lem}

\begin{proof}
By Corollary~\ref{extfin} and Lemma~\ref{uniqfin}, we may assume that $X$ is smooth and $\pi$ is \'{e}tale.
Let $\Pi:=\pi_{1}(X)$ and $\Pi':=\pi_{1}(X')$.
Then $\pi$ corresponds to the inclusion $\pi_{*}\colon \Pi' \hookrightarrow \Pi$.
Let $\Pi'':=\cap_{\gamma\in \Pi}\gamma\Pi'\gamma^{-1}$ denote the maximal normal subgroup of $\Pi$ contained in $\Pi'$.
By the Galois theory of covering maps, the desired variety $X''$ can be obtained by the \'{e}tale quotient $X'':=\widetilde{X}/\Pi''$ of the universal cover $\widetilde{X}$ of $X$ by the action of $\Pi''$.
\end{proof}

\begin{lem}[Existence of Galois alterations] \label{Galalt}
Let $\pi\colon X'\to X$ be an alteration between normal analytic varieties.
Then there exists a Galois alteration $\varphi\colon X''\to X$ which factors through $\pi$.
\end{lem}

\begin{proof}
We take the Stein factorization $X'\to Y'\xrightarrow{\pi'} X$ of $\pi$.
By Lemma~\ref{Galclo}, there exists the Galois closure $\varphi'\colon Y''\to X$ of $\pi'$.
Let $G$ denote the Galois group of $\varphi'$.
For each $g\in G$, we define a meromorphic map $\psi_{g}\colon Y''\dasharrow X'$ by the composite of meromorphic maps
$$
Y''\xrightarrow{g} Y'' \to Y'\leftarrow X'.
$$
The product $\prod_{g\in G}\psi_{g}\colon Y''\dasharrow \prod_{g\in G}X'$ is also a meromorphic map.
Let $\overline{\Gamma}$ denote the closure of the graph $\Gamma$ of $\prod_{g\in G}\psi_{g}$ in $Y'' \times \prod_{g\in G}X'$.
Then $G$ acts on $\overline{\Gamma}$ which is equivariant with respect to the projection $\overline{\Gamma}\to Y''$.
Indeed, we define the action by $g\cdot (x, (x_{h})_{h\in G}):=(gx, (x_{hg})_{h\in G})$ for $g\in G$ and $(x, (x_{h})_{h\in G})\in \overline{\Gamma}$.
Let $X''$ denote the normalization of $\overline{\Gamma}$.
Then $G$ acts on $X''$ and the composition $X''\to \overline{\Gamma}\to Y''$ is $G$-equivariant and bimeromorphic.
Thus the induced morphism $X''/G\to X=Y''/G$ is bimeromorphic.
Hence the composition $\varphi\colon X''\to X''/G \to X$ is a $G$-alteration.
The composition
$$
\psi\colon X''\to \overline{\Gamma}\to \prod_{g\in G}X'\xrightarrow{\mathrm{pr}_{e}} X'
$$
satisfies $\varphi=\pi\circ \psi$.
\end{proof}

\begin{lem} \label{equivresol}
Let $X$ be an analytic variety and $G$ a finite group acting on $X$.
Let $\varphi\colon X'\to X$ be a bimeromorphic morphism from an analytic variety $X'$.
Then there exists a $G$-equivariant resolution $\widetilde{X}\to X$ which factors through $\varphi$.
\end{lem}

\begin{proof}
By the same proof as in Lemma~\ref{Galalt}, there exists a $G$-equivariant bimeromorphic morphism $X''\to X$ which factors through $\varphi$.
Taking a $G$-equivariant resolution $\widetilde{X}\to X''$ (cf.\ \cite{BiMi}), we conclude the proof. 
\end{proof}

\begin{lem}[Kawamata's covering trick~I] \label{Kawamatacov}
Let $B$ be a smooth analytic variety which is projective over a Stein space $S$.
Let $W$ be any Stein compact subset of $S$.
Let $D$ be a reduced simple normal crossing divisor on $B$ and $D=\sum_{i}D_{i}$ the irreducible decomposition.
For each $D_{i}$, a positive integer $m_{i}$ is given.
Then after shrinking $S$ around $W$, there exists a finite abelian covering $\pi\colon B'\to B$ from a smooth analytic variety $B'$ branched along a simple normal crossing divisor such that the ramification index at each $D_{i}$ equals $m_{i}$.
\end{lem}

\begin{proof}
The proof is similar to {\cite[\S 5.3]{AbKa}}.
For the readers' convenience, we sketch the proof.
After shrinking $S$ around $W$,
we may assume that $D$ has finitely many irreducible components, say $D_{1},\ldots, D_{k}$.
Let $\mathcal{L}$ be a relatively ample sheaf on $B$ over $S$.
Then after shrinking $S$ around $W$, the sheaf $\mathcal{L}^{\otimes m}(-D_{i})$ is relatively very ample over $B$ for each $i=1,\ldots,k$ and sufficiently large and divisible $m$ (see the proof of Lemma~\ref{projemb}).
By Bertini's theorem {\cite[(II.5)~Theorem]{Mir}}, we can take $n:=\dim B$ smooth divisors $H_{i1},\ldots H_{in}$ on $B$ defined by general global sections of $\mathcal{L}^{\otimes m}(-D_{i})$ such that the intersection $\cap_{j=1}^{n}H_{ij}\cap D_{i}$ is empty and the divisor $\sum_{i=1}^{k}D_{i}+\sum_{i,j}H_{ij}$ is reduced and simple normal crossing.
Let $B_{ij}\to B$ be the standard cyclic covering of degree $m_{i}$ branched along $D_{i}+H_{ij}$ defined in the associated line bundle of $\mathcal{L}^{\otimes m/m_{i}}$.
Then the fiber product $B_{11}\times_{B}B_{12}\times_{B}\cdots \times_{B}B_{kn}$ is non-singular by a simple calculation of the defining equation and the projection
$$
\pi\colon B':=B_{11}\times_{B}B_{12}\times_{B}\cdots \times_{B}B_{kn}\to B
$$
is a desired abelian covering.
\end{proof}

More generally, the following strengthened version of Lemma~\ref{Kawamatacov} holds:
\begin{lem}[Kawamata's covering trick~II] \label{latticealt}
Let $(B, \Delta_{B})$ be a simplicial strict toroidal pair with the associated polyhedral complex $\Sigma_{(B, \Delta_{B})}$.
Assume that $B$ is projective over a Stein space $S$ and fix any Stein compact subset $W$ of $S$.
Let $\Sigma'\to \Sigma_{(B, \Delta_{B})}$ be a lattice alteration from a non-singular polyhedral complex $\Sigma'$.
Then after shrinking $S$ around $W$, there exists a finite abelian covering $\pi\colon B'\to B$ such that
the following hold:

\smallskip

\noindent
(1) $(B',\Delta_{B'}:=\pi^{-1}(\Delta_{B})_{\mathrm{red}})$ is a non-singular strict toroidal pair.

\smallskip

\noindent
(2) There exists a morphism $\Sigma_{(B', \Delta_{B'})}\to \Sigma'$ which is isomorphic at each cone and each lattice in $\Sigma_{(B', \Delta_{B'})}$.

\smallskip

\noindent
(3) There exists an effective relatively ample divisor $H$ on $B$ over $S$ such that $(B, \Delta_{B}+H)$ and $(B', \Delta_{B'}+H')$ are strict toroidal, $\pi$ is toroidal between them and the following commutative diagram holds: 
$$
\xymatrix{
\Sigma_{(B', \Delta_{B'})} \ar[r] \ar[d]  & \Sigma' \ar[r]  & \Sigma_{(B, \Delta_{B})} \ar[d] \\
\Sigma_{(B', \Delta_{B'}+H')} \ar[rr]^{\varphi_{\pi}} &  & \Sigma_{(B, \Delta_{B}+H)},
}
$$
where the vertical maps are natural inclusions and $H':=\pi^{-1}(H)_{\mathrm{red}}$.

\smallskip

\noindent
(4) Let $f\colon (X, \Delta_{X})\to (B, \Delta_{B})$ be a toroidal morphism from a strict toroidal pair.
Let $X'$ be a normalization of the fiber product $X\times_{B}B'$ and $\pi'\colon X'\to X$, $f'\colon X'\to B'$ the projections.
Then $(X', \Delta_{X'}:=\pi'^{-1}(\Delta_{X})_{\mathrm{red}})$ is strict toroidal, $f'\colon (X', \Delta_{X'})\to (B', \Delta_{B'})$ is toroidal and there exists a morphism 
$\Sigma_{(X', \Delta_{X'})}\to \Sigma_{(X, \Delta_{X})}$ such that the following diagram is commutative and Cartesian:
$$
\xymatrix{
\Sigma_{(X', \Delta_{X'})} \ar[rr] \ar[d]_{\varphi_{f'}}  &  & \Sigma_{(X, \Delta_{X})} \ar[d]^{\varphi_{f}} \\
\Sigma_{(B', \Delta_{B'})} \ar[r] & \Sigma' \ar[r] & \Sigma_{(B, \Delta_{B})}.
}
$$
In particular, the morphism $\Sigma_{(X', \Delta_{X'})}\to \Sigma_{(X, \Delta_{X})}\times_{\Sigma_{(B, \Delta_{B})}}\Sigma'$ is isomorphic at each cone and each lattice in $\Sigma_{(X', \Delta_{X'})}$.
\end{lem}

\begin{proof}
The proof is similar to the arguments in {\cite[\S 7]{AbKa}} and Lemma~\ref{Kawamatacov}. 
\end{proof}


\subsection{Family of nodal curves}

\begin{defn}[Prestable curve and stable curve]
Let $f\colon X\to B$ be a morphism of complex analytic spaces
and $\sigma_{1},\ldots,\sigma_{r}$ sections of $f$.
We call $f\colon (X,\sigma_{1},\ldots,\sigma_{r})\to B$ an {\em $r$-pointed prestable curve over $B$} if $f$ is a proper flat morphism  all of whose fibers are nodal curves and the sections $\{\sigma_{i}\}_{i=1}^{r}$ are pairwise disjoint and contained in the smooth locus of $f$ in $X$.
If moreover every fiber $(f^{-1}(p),\sigma_{1}(p),\ldots,\sigma_{r}(p))$ is an $r$-pointed stable curve, it is called an {\em $r$-pointed stable curve over $B$}.

\end{defn}

\begin{lem} \label{prestablequot}
Let $f\colon (X, \sigma_{1},\ldots, \sigma_{r})\to B$ be an $r$-pointed prestable curve over a normal analytic variety $B$.
Let $G$ be a finite group acting on $X$ over $B$ such that $f$ is $G$-invariant and the action permutes the sections $\{\sigma_{i}\}_{i=1}^{r}$.
Then the quotient $(X/G, \{\sigma_{i}\}_{i=1}^{r}/G)$ is a pointed prestable curve over $B$.
\end{lem}

\begin{proof}
The proof of {\cite[Lemma~4.3]{deJ}} also works in the complex analytic setting.
\end{proof}

\begin{lem} \label{prestabletoroidal}
Let $(B, \Delta_{B})$ be a non-singular toroidal pair.  
Let $f\colon (X,\sigma_{1},\ldots,\sigma_{r})\to B$ be an $r$-pointed prestable curve over $B$ which is smooth over $B\setminus \Delta_{B}$.
Then the following hold:

\smallskip

\noindent
(1) The pair $(X, \Delta_{X}:=f^{-1}(\Delta_{B})+\sum_{i=1}^{r}\sigma_{i})$ is toroidal and $f$ is a toroidal morphism.

\smallskip

\noindent
(2) Assume that a finite group $G$ acts on $(B, \Delta_{B})$ toroidally and its action lifts to the action on $(X, \Delta_{X})$.
Let $\beta\colon X'\to X$ denote the normalized blowing up along the singular locus of $f$ and put $\Delta_{X'}:=\beta^{-1}(\Delta_{X})$.
Then $(X', \Delta_{X'})$ is a toroidal pair and the lifting action of $G$ on $(X', \Delta_{X'})$ is pre-toroidal.
\end{lem}

\begin{proof}
The claim (1) follows easily from calculations of local coordinates.
The claim (2) follows from {\cite[\S 1.4]{AbdJ}}.
\end{proof}

\section{Semistable reduction for curves} \label{sec3}

In this section, we will show a semistable reduction theorem for families of curves in the complex analytic setting following \cite{deJ}.

\begin{thm}[Semistable reduction for curves] \label{semistredcurve}
Let $f\colon X\to B$ be a projective surjective morphism between analytic varieties
whose general fibers are connected curves.
Let $Z$ be a proper analytic subset of $X$.
Assume that $B$ is projective over a Stein space $S$. 
Let $W$ be any Stein compact subset of $S$.
Then after shrinking $S$ around $W$, there exist a projective alteration $B'\to B$ from a smooth variety $B'$, 
a projective modification $X'\to X\times_{B} B'$ from a normal variety $X'$ onto the main component of $X\times_{B} B'$ and finite disjoint sections $\sigma_{1},\ldots, \sigma_{r}$ of the projection $f'\colon X'\to B'$ such that the following hold:

\smallskip

\noindent
(i) The proper transform on $X'$ of the $f$-horizontal part of $Z$ is the union of $\sigma_{1},\ldots, \sigma_{r}$.

\smallskip

\noindent
(ii) $f'\colon (X', \sigma_{1},\ldots,\sigma_{r})\to B'$ is an $r$-pointed prestable curve over $B'$. 

\smallskip

\noindent
(iii) $B'\to B$ is a Galois alteration with Galois group $G$ and $G$ acts on $X'$ faithfully such that $X'\to X$ is a $G$-alteration, $f'$ is $G$-equivariant and the action permutes the sections $\{\sigma_{i}\}_{i=1}^{r}$.
\end{thm}

To prove Theorem~\ref{semistredcurve}, taking the blowing up along $Z$, we may assume that $Z$ is an effective Cartier divisor.
Taking the flattening of $f$ (\cite{Hir}) and the base change by a resolution of $B$, 
we may also assume that $f$ is flat and $B$ is smooth.

\begin{lem}[Genus goes up] \label{genusgoesup}
Let $f\colon X\to B$ be a projective surjective equidimensional morphism between normal analytic varieties 
whose fibers are connected curves.
Assume that $B$ is projective over a Stein space $S$ and fix any Stein compact subset $W$ of $S$.
Let $n$ be any natural number.
Then after shrinking $S$ around $W$, there exists a finite Galois covering $\pi\colon Y\to X$ from a normal analytic variety $Y$ such that any irreducible $f\circ \pi$-vertical curve has geometric genus greater than $n$.
\end{lem}

\begin{proof}
By Lemma~\ref{projemb}, there exists a closed immersion $X\hookrightarrow \mathbb{P}^{N}\times B$ over $B$ for some $N$ after shrinking $S$ around $W$.
Now we consider linear projections $\mathbb{P}^{N}\dasharrow \mathbb{P}^{1}$.
For each codimension $2$ plane $P$ in $\mathbb{P}^N$, we denote by $\pi_{P}\colon \mathbb{P}^N\dasharrow \mathbb{P}^1$ the linear projection from $P$.
Let $\varphi_{P}\colon X\dasharrow \mathbb{P}^1\times B$ denote the composition $X\hookrightarrow \mathbb{P}^N\times B\overset{\pi_{P}\times \mathrm{id}}{\dashrightarrow} \mathbb{P}^1\times B$.
It is a meromorphic map whose indeterminacy locus is equal to $X\cap (P\times B)$.
We use the following lemma:

\begin{lem} \label{linprojfin}
Let $f\colon X\to B$ be a projective surjective equidimensional morphism between analytic spaces
whose fibers are connected curves.
Assume that $d:=\dim B$ is finite and there exists a closed immersion $X\hookrightarrow \mathbb{P}^{N}\times B$ over $B$ for some $N>1$.
Then there exist finitely many planes $P_{1},\ldots,P_{d}$ in $\mathbb{P}^N$ of codimension $2$ such that the following holds:
For any $b\in B$, there exists $i=1,\ldots,d$ such that the restriction of the linear projection from $P_{i}$ to the fiber $X_{b}$ defines a finite morphism $\varphi_{P_i}|_{X_{b}}\colon X_{b}\to \mathbb{P}^1$.
\end{lem}

\begin{proof}
We show the claim by induction on $d=\dim B$.
When $\dim B=0$, the claim is trivial.
Assume that $B$ is positive dimensional.
Let $\mathcal{H}$ denote the analytic subspace of $B\times \mathrm{Gr}(N-2, \mathbb{P}^{N})$ consisting of points $(b, P)$ such that $\varphi_{P}|_{X_{b}}$ is not a finite morphism.
Note that the projection $\mathrm{pr}_{1}\colon \mathcal{H}\to B$ is proper since the Grassmannian $\mathrm{Gr}(N-2, \mathbb{P}^{N})$ is projective.
Hence $B_{1}:=\mathrm{pr}_{1}(\mathrm{pr}_{2}^{-1}(P_1))$ is an analytic subspace of $B$ for any $P_{1}\in \mathrm{Gr}(N-2, \mathbb{P}^{N})$.
By taking $P_{1}$ very general, we may assume that $\dim B_{1}<\dim B$.
By using the inductive assumption for $f^{-1}(B_{1})\to B_{1}$, there exist $P_{2}\ldots,P_{d}\in \mathrm{Gr}(N-2, \mathbb{P}^{N})$ such that for any $b\in B_{1}$, the restriction $\varphi_{P_{i}}|_{X_{b}}\colon X_{b}\to \mathbb{P}^1$ is finite for some $i=2,\ldots, d$.
For any $b\in B\setminus B_{1}$, the restriction $\varphi_{P_{1}}|_{X_{b}}\colon X_{b}\to \mathbb{P}^1$ is finite by the definition of $\mathcal{H}$.
\end{proof}

{\it Proof of Lemma \ref{genusgoesup} continued.}
Let $P_{1},\ldots, P_{d}$ be as in Lemma~\ref{linprojfin}.
For each $i$, we take a resolution of indeterminacy $\widetilde{\varphi}_{P_{i}}\colon \widetilde{X}_{i}\to \mathbb{P}^{1}\times B$ of $\varphi_{P_{i}}$ and a Galois covering $\pi_{i}\colon C_{i}\to \mathbb{P}^{1}$ with Galois group $G_{i}$ from a smooth curve $C_{i}$ of genus greater than $n$.
Let $\widetilde{Y}_{i}\to \widetilde{X}_{i}$ denote the $G_{i}$-covering defined as the normalized base change of $\pi_{i}$ by the composition $\mathrm{pr}_{1}\circ \widetilde{\varphi}_{P_{i}}\colon \widetilde{X}_{i} \to \mathbb{P}^1$
and take the Stein factorization $\widetilde{Y}_{i}\to Y_{i}\to X$ of the composition $\widetilde{Y}_{i}\to \widetilde{X}_{i}\to X$.
Let $Y$ be the normalization of the fiber product $Y_{1}\times_{X}\cdots\times_{X}Y_{d}$ and put $H:=G_{1}\times\cdots \times G_{d}$.
Then the projection $\pi\colon Y\to X$ is an $H$-covering and the composition $f\circ \pi\colon Y\to B$ satisfies the desired property.
Indeed, for any irreducible component $C$ in a fiber $Y_{b}$, we can find the finite morphisms 
$$
C\to Y_{i,b}=X_{b}\times_{\mathbb{P}^{1}}C_{i} \to C_{i}
$$
by taking $i$ such that $\varphi_{P_{i}}|_{X_{b}}\colon X_{b}\to \mathbb{P}^{1}$ is finite.
Hence the geometric genus of $C$ is greater than $n$ since so is $C_{i}$.
\end{proof}

Let $\pi\colon Y\to X$ be a Galois covering as in Lemma~\ref{genusgoesup} with $n=1$ and $H$ its Galois group.
We may assume that $S$ remains Stein since $W$ is Stein compact.
Let $Z_{Y}$ denote the pullback on $Y$ of the effective Cartier divisor $Z$.
Then we can take an invariant stable reduction of $Y\to X\to B$ as follows:

\begin{prop}[Invariant stable reduction] \label{invstablered}
Let $g\colon Y\to B$ be a projective equidimensional morphism between normal analytic varieties whose fibers are connected and consist of curves with geometric genus greater than one.
Let $Z_Y$ be an effective Cartier divisor on $Y$.
Assume that $B$ is projective over a Stein space $S$ and fix any Stein compact subset $W$ of $S$. 
Further assume that a finite group $H$ acts on $Y$ such that $g$ and $Z_{Y}$ are $H$-invariant.
Then after shrinking $S$ around $W$, there exist a projective alteration $B'\to B$ from a smooth variety $B'$, 
a projective modification $Y'\to Y\times_{B} B'$ from a normal variety $Y'$ onto the main component of $Y\times_{B} B'$ and finite disjoint sections $\tau_{1},\ldots, \tau_{s}$ of the projection $g'\colon Y'\to B'$ such that the following hold:

\smallskip

\noindent
(i) The proper transform on $Y'$ of the $g$-horizontal part of $Z_Y$ is the union of $\tau_{1},\ldots, \tau_{s}$.

\smallskip

\noindent
(ii) $g'\colon (Y', \tau_{1},\ldots,\tau_{s})\to B'$ is an $s$-pointed stable curve over $B'$. 

\smallskip

\noindent
(iii) $B'\to B$ is a Galois alteration with Galois group $G$ and $G$ acts on $Y'$ faithfully such that $Y'\to Y$ is a $G$-alteration, $g'$ is $G$-equivariant and the action permutes the sections $\{\tau_{i}\}_{i=1}^{s}$.

\smallskip

\noindent
(iv) $H$ acts on $Y'$ such that $Y'\to Y$ is $H$-equivariant, $g'$ is $H$-invariant and the action permutes the sections $\{\tau_{i}\}_{i=1}^{s}$.
\end{prop}

\begin{proof}
By taking normalized base change from a resolution of each horizontal component of $Z_{Y}$ to $B$ and applying Lemma~\ref{Galalt},
we may assume that $B$ is smooth and the horizontal part of $Z_{Y}$ is the sum of some sections $\sigma_{1},\ldots, \sigma_{s}$ of $g$.
Let $D$ be the minimal analytic subset of $B$ such that the restriction of $g\colon (Y, \sigma_{1},\ldots, \sigma_{s})\to B$ over $B\setminus D$ is a family of $s$-pointed smooth curves.
By taking a further modification of $B$, we may also assume that $D$ is simple normal crossings.

Taking successive blowing-ups of $Y$ along irreducible components of the singular loci of $g^{-1}(D)_{\mathrm{red}}$ and $\cup_{i=1}^{s}\sigma_{i}$ dominating an irreducible component of $D$, 
We obtain a projective modification $\widetilde{Y}\to Y$ such that the following property holds:
There exist a Zariski open subspace $V\subseteq B$ the complement of which has codimension at least $2$ and the sections $\widetilde{\sigma}_{1},\ldots, \widetilde{\sigma}_{s}$ of the composition $\widetilde{g}\colon \widetilde{Y}\to Y\to B$ lifting of $\sigma_{1},\ldots,\sigma_{s}$ such that the tuple $(\widetilde{g}^{-1}(p)_{\mathrm{red}}, \widetilde{\sigma}_{1}(p),\ldots, \widetilde{\sigma}_{s}(p))$ is an $s$-pointed prestable curve for each $p\in V$.
For an irreducible components $D_{\lambda}$ of $D$, let $m_{\lambda}$ denote the least common multiple of the geometric multiplicities of components of $\widetilde{g}^{-1}(D_{\lambda}\cap V)$. 
By using Lemma~\ref{Kawamatacov}, after shrinking $S$ around $W$, we can take a finite abelian covering $\pi \colon B_{1}\to B$ from a smooth variety $B_{1}$ branched along at most $D+A$ for some general ample divisor $A$ satisfying the following property:
Let $Y_{1}$ denote the normalization of the main component of $\widetilde{Y}\times_{B}B_{1}$ with the projection $g_{1}\colon Y_{1}\to B_{1}$ and $U:=\pi^{-1}(V)$.
Then there exist sections $\sigma_{U, 1},\ldots, \sigma_{U, s}$ of $g_{1}$ over $U$ which are lifts of $\widetilde{\sigma}_{1},\ldots, \widetilde{\sigma}_{s}$ such that
$(Y_{1, U}:=g_{1}^{-1}(U), \sigma_{U, 1},\ldots, \sigma_{U, s})$ is an $s$-pointed prestable curve over $U$.
Let $g'_{U}\colon (Y'_{1, U}, \sigma'_{U, 1},\ldots, \sigma'_{U, s})\to U$ denote the stable model of $g_{U}:=g_{1}|_{Y_{1,U}}\colon (Y_{1, U}, \sigma_{U, 1},\ldots, \sigma_{U, s})\to U$ with the contraction
$$
Y_{1, U}\to Y'_{1, U}:=\mathrm{Projan}_{B_{1}}\bigoplus_{n\ge 0}g_{U*}\omega_{g_{U}}^{\otimes n}(n\sum_{i=1}^{s}\sigma_{U, i}).
$$ 
Since the complement $B_{1}\setminus U$ has codimension at least $2$ and the discriminant locus of $g_{1}$ is simple normal crossings, it follows from {\cite[Corollary~5.5]{dJOo}} (applying to each local ring of $B_{1}$) that 
the $s$-pointed stable curve $(Y'_{1, U}, \sigma'_{U, 1},\ldots, \sigma'_{U, s})$ over $U$ extends to an $s$-pointed stable curve $g'_{1}\colon (Y'_{1}, \sigma'_{1},\ldots, \sigma'_{s})\to B_{1}$, where $Y'_{1}$ is realized as
$$
Y'_{1}=\mathrm{Projan}_{B_{1}}\bigoplus_{n\ge 0}i_{U*}g_{U*}\omega_{g_{U}}^{\otimes n}(n\sum_{i=1}^{s}\sigma_{U, i})\cong \mathrm{Projan}_{B_{1}}\bigoplus_{n\ge 0}i_{U*}g'_{U*}\omega_{g'_{U}}^{\otimes n}(n\sum_{i=1}^{s}\sigma'_{U, i}),
$$
where $i_{U}\colon U\hookrightarrow B_{1}$ is the inclusion.

Next we show that the map
$$
\varphi\colon Y'_{1}\hookleftarrow Y'_{1, U}\leftarrow Y_{1, U}\to Y\times_{B}U\hookrightarrow Y\times_{B}B_{1}
$$
extends to a bimeromorphic morphism $Y'_{1}\to Y\times_{B}B_{1}$.
To prove this, we use the following two lemmas:

\begin{lem} \label{morextn}
Let $f\colon X\to B$ and $f'\colon X'\to B$ be proper surjective equidimensional morphisms between normal analytic varieties whose fibers are connected curves.
Assume that $f'$ is prestable and any component of any fiber of $f$ has positive geometric genus.
Then any meromorphic map $X'\dasharrow X$ over $B$ is a morphism.
\end{lem}

\begin{proof}
Let $\varphi\colon X'\dasharrow X$ be a meromorphic map.
Then the closure $\overline{\Gamma}\subseteq X'\times_{B}X$ of the graph of $\varphi$ is an analytic subset and $\varphi$ can be written as
$$
\varphi\colon X'\xleftarrow{\rho'} \overline{\Gamma}\xrightarrow{\rho} X,
$$
where the first projection $\rho'$ is a bimeromorphic morphism.
To prove the claim, it is sufficient to show that $\rho'$ is an isomorphism.
By Zariski's main theorem and the normality of $X'$, it suffices to show that any fiber of $\rho'$ is $0$-dimensional. 
For any $x\in X'$, the fiber $\rho'^{-1}(x)$ is contained in $x\times_{B}X\cong f^{-1}(f'(x))$.
Hence each positive dimensional component of $\rho'^{-1}(x)$ is a curve of positive geometric genus.
On the other hand, taking a normalization $C\to B$ of a general curve on $B$ passing through $f'(x)$ and the base change $\overline{\Gamma}_{C}\to X'_{C}\to C$ of the composition $\overline{\Gamma}\to X'\to B$,
the surface $X'_{C}$ has at most rational singularities since $f'$ is prestable. 
Since any positive dimensional component of $\rho'^{-1}(x)$ is exceptional over $X'_{C}$, it is a rational curve.
Hence $\rho'^{-1}(x)$ has no positive dimensional components.
\end{proof}

\begin{lem} \label{meroextn}
Let $f\colon X\to B$ and $f'\colon X'\to B$ be proper surjective equidimensional morphisms between analytic varieties whose fibers are connected curves.
Let $U\subseteq B$ be a Zariski open subset whose complement $B\setminus U$ has codimension at least $2$.
Then any morphism $\varphi_{U}\colon f^{-1}(U)\to f'^{-1}(U)$ over $U$ extends to a meromorphic map $\varphi\colon X\dasharrow X'$ over $B$.
\end{lem}

\begin{proof}
It suffices to show that the closure $\overline{\Gamma}$ in $X'\times_{B} X$ of the graph $\Gamma\subseteq X'\times_{B}X\times_{B}U$ of $\varphi_{U}$ is an analytic subset.
First note that the complement $X'\times_{B}X \setminus (X'\times_{B}X\times_{B}U)$ has codimension at least $2$ since so is the complement $B\setminus U$.
Hence the dimension of $X'\times_{B}X\setminus (X'\times_{B}X\times_{B}U)$ is less than $n:=\dim X=\dim \Gamma$.
Thus the claim follows from Remmert-Stein's theorem {\cite[Theorem~4.6]{Shi}}.
\end{proof}

{\it Proof of Proposition \ref{invstablered} continued.} 
First note that the restriction $\varphi_{U}\colon Y'_{1, U} \dasharrow Y\times_{B}U$ over $U$ is a bimeromorphic map.
Applying Lemma~\ref{morextn}, $\varphi_{U}$ is a morphism.
From Lemma~\ref{meroextn}, 
$\varphi_{U}$ extends to a bimeromorphic map $\varphi\colon Y'_{1} \dasharrow Y\times_{B}B_{1}$.
Applying Lemma~\ref{morextn} again, $\varphi$ becomes a morphism.

Let $G$ denote the Galois group of the Galois alteration $B_{1}\to B$.
Then both $G$ and $H$ act on $Y\times_{B}B_{1}$ naturally.
Now we show that $G$ and $H$ also act on $Y'$.
To prove this, we take a resolution $\widetilde{Y}_{1}\to Y\times_{B}B_{1}$ which is $(G, H)$-equivariant and factors through $\varphi\colon Y'_{1}\to Y\times_{B}B_{1}$ by Lemma~\ref{equivresol}, denoted by $\rho \colon \widetilde{Y}_{1}\to Y'_{1}$.
Let $\widetilde{Z}$ and $Z'$ be the pullbacks of the horizontal part of $Z_{Y}$ to $\widetilde{Y}_{1}$ and $Y'_{1}$, respectively.
By construction, $\widetilde{Z}$ is $(G, H)$-invariant and $Z'$ coincides with the disjoint union of $\sigma'_{1},\ldots, \sigma'_{s}$.
Since $\widetilde{Z}$ is the pullback of $Z'$ and $Y'_{1}$ has at most rational singularities,
we have $\rho_{*}\omega_{\widetilde{g}_{1}}^{\otimes n}(n\widetilde{Z})\cong \omega_{g'_{1}}^{\otimes n}(nZ')$ for any $n$,
where we put $\widetilde{g}_{1}:=g'_{1}\circ \rho \colon \widetilde{Y}_{1}\to B_{1}$.
Hence we have
$$
Y'_{1}\cong \mathrm{Projan}_{B_1}\bigoplus_{n\ge 0}g'_{1*}\omega_{g'_{1}}^{\otimes n}(nZ')\cong \mathrm{Projan}_{B_1}\bigoplus_{n\ge 0}\widetilde{g}_{1*}\omega_{\widetilde{g}_{1}}^{\otimes n}(n\widetilde{Z}).
$$
In particular, $g'_{1}\colon (Y'_{1}, \sigma'_{1},\ldots,\sigma'_{s})\to B_{1}$ is the relative canonical model of $\widetilde{g}_{1}\colon (\widetilde{Y}_{1}, \widetilde{Z})\to B_{1}$.
Thus $G$ and $H$ acts on $Y'_{1}$ which permutes the sections $\{\sigma'_{i}\}_{i=1}^{s}$ such that $\rho$ is $(G, H)$-equivariant and $g'_{1}$ is $G$-equivariant and $H$-invariant.
Letting $g':=g'_{1}$, $Y':=Y'_{1}$, $B':=B_{1}$ and $\tau_{i}:=\sigma'_{i}$, we conclude the proof of Proposition~\ref{invstablered}.
\end{proof}

Now we are ready to prove Theorem~\ref{semistredcurve}.

\begin{proof}[Proof of Theorem~\ref{semistredcurve}]
Let $B'\to B$ and $g'\colon (Y', \tau_{1},\ldots \tau_{s})\to B'$ be the $G$-alteration and the $s$-pointed stable curve over $B'$ obtained by applying Proposition~\ref{invstablered} to $g:=f\circ \pi\colon Y\to X\to B$, respectively.
Then the quotient $X':=Y'/H$ and the sections $\{\sigma_{j}\}_{j=1}^{r}:=\{\tau_{i}\}_{i=1}^{s}/H$ define an $r$-pointed prestable curve $f'\colon (X', \sigma_{1},\ldots, \sigma_{r})\to B$ by Lemma~\ref{prestablequot}.
Since the actions of $G$ and $H$ on $Y'$ are commute, $G$ also acts on $X'$ such that $f'$ is $G$-equivariant, $X'\to X$ is a $G$-alteration and the action permutes $\{\sigma_{i}\}_{i=1}^{r}$.
By construction, the proper transform on $X'$ of the horizontal part of $Z$ coincides with the disjoint union of $\sigma_{1}\ldots, \sigma_{r}$.
Hence we conclude the proof of Theorem~\ref{semistredcurve}.
\end{proof}

\section{Proof of main theorem} \label{sec4}

\subsection{Toroidal reduction}

In this subsection, we prove the toroidal reduction theorem, which is a complex analytic version of {\cite[Theorem~2.1]{AbKa}}.

\begin{thm}[Toroidal reduction] \label{toroidalred}
Let $f\colon X\to B$ be a projective surjective morphism between analytic varieties.
Let $Z$ be a proper analytic subset of $X$.
Assume that $B$ is projective over a Stein space $S$ and fix any Stein compact subset $W$ of $S$.
Then after shrinking $S$ around $W$, there exist projective resolutions $m_{B}\colon B'\to B$, $m_{X'}\colon X'\to X$, a morphism $f'\colon X'\to B'$ with
$m_{B}\circ f'=f\circ m_{X}$ and reduced divisors $\Delta_{B'}\subseteq B'$ and $\Delta_{X'}\subseteq X'$ such that the following holds.

\smallskip

\noindent
(i) $(B', \Delta_{B'})$ and $(X', \Delta_{X'})$ are strict toroidal pairs.

\smallskip

\noindent
(ii) $f'$ is a toroidal morphism.

\smallskip

\noindent
(iii) The union of $m_{X}^{-1}(Z)$ and the $m_{X}$-exceptional set forms a divisor with simple normal crossing support contained in $\Delta_{X'}$.
\end{thm}

\begin{proof}
The proof is almost identical to the proof of {\cite[Theorem~2.1]{AbKa}} by replacing {\cite[Theorem~2.4]{deJ}} with Theorem~\ref{semistredcurve}.
For the readers' convenience, we prove this.

Let $n$ denote the relative dimension of $f$.
We prove the claim by the induction on $n$.
In the case where $n=0$, then $f$ is a projective alteration.
Replacing $\widetilde{X}$ by its projective modification, we may assume that $Z$ is an effective Cartier divisor.
From Lemma~\ref{Galalt}, We can take a Galois alteration $\widetilde{X}\to B$ which factors through $f$ and denote $G$ its Galois group.
Let $X_{1}$ denote the normalization of the base change $X\times_{B}\widetilde{X}/G$.
Then $X_{1}\to \widetilde{X}/G$ is finite since the quotient $\widetilde{X}\to \widetilde{X}/G$ factors through $X_{1}$.
We take a projective resolution $B'\to \widetilde{X}/G$ after shrinking $S$ around $W$ such that the normalized base change $\pi\colon X_{2}\to B'$ of $X_{1}\to \widetilde{X}/G$ is \'{e}tale over the complement of some simple normal crossing divisor $\Delta_{B'}$.
We may assume that the pullback of $Z$ by $X_{2}\to X$ is contained in the inverse image $\pi^{-1}(\Delta_{B'})$.
It follows from Proposition~\ref{analAbh} that the pair $(X_{2}, \Delta_{X_{2}}:=\pi^{-1}(\Delta_{B'})_{\mathrm{red}})$ is toroidal and $\pi$ is a toroidal morphism.
Taking a toroidal resolution $(X', \Delta_{X'})\to (X_{2}, \Delta_{X_{2}})$, the claim holds.

Suppose that the claim holds for the case of relative dimension $n-1$.
By Lemma~\ref{projemb}, We can take a closed immersion $X\subseteq \mathbb{P}^{N}\times B$ over $B$ for some $N$ after shrinking $S$ around $W$.
We consider a linear projection $\mathbb{P}^{N}\dasharrow \mathbb{P}^{n-1}$ from a general $(N-n)$-dimensional plane.
Replacing $X$ by the resolution of indeterminacy of the composition 
$$
X\hookrightarrow \mathbb{P}^{N}\times B\dasharrow \mathbb{P}^{n-1}\times B,
$$
we may assume that it defines a morphism $g'\colon X\to \mathbb{P}^{n-1}\times B$.
Let $X\to P\to \mathbb{P}^{n-1}\times B$ be its Stein factorization.
Applying Theorem~\ref{semistredcurve} to the morphism $g\colon X\to P$ and $Z$, after shrinking $S$ around $W$,
there exist Galois alterations $a\colon P_{1}\to P$, $\alpha\colon X_{1}\to X$ with Galois group $G$, a $G$-equivariant morphism $g_{1}\colon X_{1}\to P_{1}$ and disjoint sections $\sigma_{1},\ldots, \sigma_{r}$ of $g_{1}$ such that $a\circ g_{1}=g\circ \alpha$, the union of $\sigma_{1},\ldots,\sigma_{r}$ is the proper transform of the horizontal part of $Z$ and 
$$
g_{1}\colon (X_{1},\sigma_{1},\ldots,\sigma_{r})\to P_{1}
$$
is an $r$-pointed prestable curve.
Replacing $X$, $P$ and $Z$ by $X_{1}/G$, $P_{1}/G$ and $\alpha^{-1}(Z)/G$, we may assume that $X=X_{1}/G$ and $P=P_{1}/G$.
Let $Z_{P}\subseteq P$ denote the union of the loci over which $X_{1}$, $P_{1}$ or $\alpha^{-1}(Z)$ are not smooth.
By using the inductive assumption to the projection $h\colon P\to B$ and $Z_{P}$, there exist projective resolutions $m_{B}\colon B'\to B$, $m_{P}\colon P'\to P$, a morphism $h'\colon P'\to B'$ with $m_{B}\circ h'=h\circ m_{P}$ and reduced divisors $\Delta_{B'}\subseteq B'$ and $\Delta_{P'}\subseteq P'$ such that the condition (i), (ii) and (iii) in Theorem~\ref{toroidalred} hold.
Replacing $P$, $B$ by $P'$, $B'$ and $X$, $X_{1}$, $P_{1}$, $Z$, $\sigma_{i}$ by these pullbacks to $P'$, we may assume that $(B, \Delta_{B})$ and $(P, \Delta_{P})$ are non-singular strict toroidal pairs, $h\colon P\to B$ is a toroidal morphism and $X_{1}$, $P_{1}$ and $\alpha^{-1}(Z)$ are smooth over $P\setminus \Delta_{P}$.
Since the $G$-covering $a\colon P_{1}\to P$ is finite \'{e}tale over $P\setminus \Delta_{P}$ and $P_{1}$ is normal, the pair $(P_{1},\Delta_{P_{1}}:=a^{-1}(\Delta_{P})_{\mathrm{red}})$ is toroidal on which $G$ acts toroidally and $a$ is a toroidal morphism by Proposition~\ref{analAbh}.
Since $g_{1}\colon (X_{1},\sigma_{1},\ldots,\sigma_{r})\to P_{1}$ is an $r$-pointed prestable curve and smooth over $P_{1}\setminus \Delta_{P_{1}}$, the pair
$$
(X_{1},\Delta_{X_{1}}:=g_{1}^{-1}(\Delta_{P_{1}})+\sum_{i=1}^{r}\sigma_{i})
$$
is toroidal and $g_{1}$ is a toroidal morphism by Lemma~\ref{prestabletoroidal}~(1).
Let $\beta\colon X_{2}\to X_{1}$ denote the normalized blowing up along the singular locus of $g_{1}$
and $\Delta_{X_{2}}$ the inverse image of $\Delta_{X_{1}}$ by $\beta$.
Then $(X_{2},\Delta_{X_{2}})$ is also toroidal and the lifting action of $G$ on it is pre-toroidal by Lemma~\ref{prestabletoroidal}~(2).
Let $\gamma\colon X_{3}\to X_{2}$ denote the normalized blowing up along the torific ideal $I$ and $\Delta_{X_{3}}$ the inverse image of the union of $\Delta_{X_{2}}$ and the zero-locus of $I$ by $\gamma$.
Then $(X_{3}, \Delta_{X_{3}})$ is also toroidal and the lifting action of $G$ on it is toroidal by Lemma~\ref{torify}.
Let $X':=X_{3}/G$ and $\Delta_{X'}:=\Delta_{X_{3}}/G$.
Then $(X', \Delta_{X'})$ is toroidal and the quotient morphism $X'\to P=P_{1}/G$ is a toroidal morphism.
Finally, replacing $(X', \Delta_{X'})$ by its projective toroidal resolution (\cite[Chapter~III, Theorem~4.1]{KKMS}), we conclude the proof.
\end{proof}

\subsection{Equidimensional reduction}

\begin{lem} \label{equidimsubdiv}
Let $\varphi\colon \Sigma\to T$ be a morphism of rational conical polyhedral complexes.
Then there exists a projective resolution $T'\to T$ such that the induced morphism $\Sigma\times_{T}T'\to T'$ is equidimensional.
\end{lem}

\begin{proof}
The claim follows from the proofs of {\cite[Lemma~4.3]{AbKa}} and {\cite[Proposition~4.4]{AbKa}}.
\end{proof}

Combining Theorem~\ref{toroidalred} with Lemmas~\ref{subdivmodif}, ~\ref{subdivtoroidal}, ~\ref{toroidalmorproperty}~(2) and ~\ref{equidimsubdiv},
the following equidimensional reduction theorem holds:

\begin{thm}[Equidimensional reduction] \label{equidimensionalred}
Let $f\colon X\to B$ be a projective surjective morphism between analytic varieties.
Let $Z$ be a proper analytic subset of $X$.
Assume that $B$ is projective over a Stein space $S$ and fix any Stein compact subset $W$ of $S$.
Then after shrinking $S$ around $W$, there exist a projective resolutions $m_{B}\colon B'\to B$, a projective modification $m_{X'}\colon X'\to X$, a morphism $f'\colon X'\to B'$ with
$m_{B}\circ f'=f\circ m_{X}$ and reduced divisors $\Delta_{B'}\subseteq B'$ and $\Delta_{X'}\subseteq X'$ such that the following holds.

\smallskip

\noindent
(i) $(B', \Delta_{B'})$ and $(X', \Delta_{X'})$ are strict toroidal pairs.

\smallskip

\noindent
(ii) $f'$ is an equidimensional toroidal morphism.

\smallskip

\noindent
(iii) The union of $m_{X}^{-1}(Z)$ and the $m_{X}$-exceptional set forms a divisor contained in $\Delta_{X'}$.
\end{thm}

\subsection{Weakly semistable reduction}

\begin{lem}[cf.\ {\cite[Proposition~5.10]{AbKa}}] \label{weaksemistalteration}
Let $\varphi\colon \Sigma\to T$ be a morphism of rational conical polyhedral complexes.
Then there exists a projective alteration $T'\to T$ such that the induced morphism $\Sigma\times_{T}T'\to T'$ is weakly semistable.
\end{lem}

\begin{proof}
By Lemma~\ref{equidimsubdiv}, replacing $\varphi$ with its base change to a projective resolution of $T$,
we may assume that $\varphi$ is equidimensional.
For each cone $\sigma\in \Sigma$, the image $\varphi(\sigma)$ is a cone in $T$ and the image of the lattice homomorphism is a sublattice $\varphi_{\sigma}(N_{\sigma})\subseteq N_{\varphi(\sigma)}$.
We define a rational conical polyhedral complex $T'$ by the data $\{(\tau, N'_{\tau})\}_{\tau\in T}$, where the lattice $N'_{\tau}$ is defined by the intersection of $\varphi_{\sigma}(N_{\sigma})$ for all $\sigma\in \Sigma$ with $\tau=\varphi(\sigma)$.
Then the natural morphism $T'\to T$ is a lattice alteration and the induced morphism $\Sigma\times_{T}T'\to T'$ is weakly semistable.
\end{proof}

The problem for translating Lemma~\ref{weaksemistalteration} into the language of toroidal morphisms is to realize a toroidal morphism corresponding to a lattice alteration.
One (weak) answer is to use Kawamata's covering trick (Lemma~\ref{Kawamatacov}, Lemma~\ref{latticealt}).

\begin{thm}[Weak semistable reduction] \label{weaksemistred}
Let $f\colon X\to B$ be a projective surjective morphism between analytic varieties.
Let $Z$ be a proper analytic subset of $X$.
Assume that $B$ is projective over a Stein space $S$ and fix any Stein compact subset $W$ of $S$.
Then after shrinking $S$ around $W$, there exist a projective alteration $B'\to B$ from a smooth variety $B'$, a projective modification $X'\to X\times_{B}B'$ from a normal variety $X'$ onto the main component of $X\times_{B}B'$ and reduced divisors $\Delta_{B'}\subseteq B'$ and $\Delta_{X'}\subseteq X'$ such that the following hold:

\smallskip

\noindent
(i) $(B', \Delta_{B'})$ and $(X', \Delta_{X'})$ are strict toroidal pairs.

\smallskip

\noindent
(ii) The projection $f'\colon (X', \Delta_{X'})\to (B', \Delta_{B'})$ is a weakly semistable toroidal morphism.

\smallskip

\noindent
(iii) The union of the inverse image of $Z$ and the exceptional set on $X'$ over $X$ forms a divisor contained in $\Delta_{X'}$.
\end{thm}

\begin{proof}
By Theorem~\ref{equidimensionalred}, we may assume that $(B, \Delta_{B})$ is a non-singular strict toroidal pair for some divisor $\Delta_{B}$, $(X, \Delta_{X}:=Z)$ is strict toroidal and $f\colon X\to B$ is an equidimensional toroidal morphism between them.
Let $T:=\Sigma_{(B, \Delta_{B})}$ and $\Sigma:=\Sigma_{(X, \Delta_{X})}$ denote the associated polyhedral complexes.
Then by Lemma~\ref{weaksemistalteration} and its proof, there exists a lattice alteration $T'\to T$ such that the induced morphism $\Sigma \times_{T}T'\to T'$ is weakly semistable.
Applying Lemma~\ref{latticealt} to $T'\to T$, there exists a finite abelian covering $\pi\colon B'\to B$ such that the normalization $X'$ of $X\times_{B}B'$ admits a toroidal morphism 
$$
f'\colon (X',\Delta_{X'}:=\pi'^{-1}(\Delta_{X})_{\mathrm{red}})\to (B', \Delta_{B'}:=\pi^{-1}(\Delta_{B})_{\mathrm{red}})
$$
between strict toroidal pairs, where $\pi'\colon X'\to X$ is the projection.
Moreover, $B'$ is non-singular and $f'$ is weakly semistable since $\Sigma_{(B', \Delta_{B'})}\to T'$ and $\Sigma_{(X', \Delta_{X'})}\to \Sigma \times_{T}T'$ is isomorphic at each cone and each lattice and hence $\Sigma_{(X', \Delta_{X'})}\to \Sigma_{(B', \Delta_{B'})}$ is weakly semistable.
\end{proof}

\subsection{Semistable reduction}
Recently, Adiprasito, Liu and Temkin proved the following polyhedral conjecture of Abramovich and Karu:

\begin{thm}[{\cite[Conjecture~8.4]{AbKa}}, {\cite[Theorem~2.7]{ALT}}] \label{ALTthm}
Let $\varphi\colon \Sigma\to T$ be a morphism of rational conical polyhedral complexes.
Then there exists a projective alteration $T'\to T$ and a projective subdivision $\Sigma'\to \Sigma\times_{T}T'$ such that the induced morphism $\Sigma'\to T'$ is semistable.
\end{thm}

Using the same arguments as in the proof of Theorem~\ref{weaksemistred} by replacing Lemma~\ref{weaksemistalteration} with Theorem~\ref{ALTthm}, the following semistable reduction theorem holds:

\begin{thm}[Semistable reduction] \label{semistred}
Let $f\colon X\to B$ be a projective surjective morphism between analytic varieties.
Let $Z$ be a proper analytic subset of $X$.
Assume that $B$ is projective over a Stein space $S$ and fix any Stein compact subset $W$ of $S$.
Then after shrinking $S$ around $W$, there exist a projective alterations $B'\to B$ from a smooth variety $B'$, a projective modification $X'\to X\times_{B}B'$ from a smooth variety $X'$ onto the main component of $X\times_{B}B'$ and reduced divisors $\Delta_{B'}\subseteq B'$ and $\Delta_{X'}\subseteq X'$ such that the following hold:

\smallskip

\noindent
(i) $(B', \Delta_{B'})$ and $(X', \Delta_{X'})$ are strict toroidal pairs.

\smallskip

\noindent
(ii) The projection $f'\colon (X', \Delta_{X'})\to (B', \Delta_{B'})$ is a semistable toroidal morphism.

\smallskip

\noindent
(iii) The union of the inverse image of $Z$ and the exceptional set on $X'$ over $X$ forms a divisor contained in $\Delta_{X'}$.
\end{thm}

\subsection{Reduction to the algebraic case}\label{subsecreduction}
In this subsection, we give an alternate proof of Theorem~\ref{Introsemistred} by reducing to the algebraic case, which is taught by Dan Abramovich. 

Let $f\colon X\to B$ be a projective morphism between complex analytic varieties and assume that $B$ is projective over a Stein space $S$.
Let $W$ be any Stein compact subset of $S$.
Applying Hironaka's flattening and Lemma~\ref{projemb}, we may assume that $X$ is an analytic subspace of $B\times \mathbb{P}^{N}$ which is flat over $B$ after shrinking $S$ around $W$.
This provides a morphism $\rho\colon B\to \mathrm{Hilb}(\mathbb{P}^{N})$ between analytic spaces.
Let $H$ denote the Zariski closure of its image $\rho(B)$.
Note that $H$ is a projective variety and the inclusion $\rho(B)\subseteq H$ may not be dense in the usual topology.
Let $g\colon U\to H$ denote the restriction of the universal family to $H$.
Let us take a weak semistable reduction 
$$
g_{1}\colon (U_1,\Delta_{U_{1}})\to (H_1, \Delta_{H_{1}})
$$
of $g$ in the sense of \cite{AbKa} and \cite{Kar}.
Note that the discriminant of the alteration $H_{1}\to H$ does not contain the image $\rho(B)$ since $H$ is the Zariski closure of $\rho(B)$.
Hence we can take a projective alteration  $B_{1} \to B$ from a smooth variety $B_{1}$ which factors through $B\times_{H}H_{1}$ (Indeed, take $B_{1}$ as a projective resolution of an irreducible component of $B_{1}\times_{H}H_{1}$ which dominates $B$).
By taking further modifications, we may assume that the pullback of $\Delta_{H_{1}}$ to $B_{1}$ is normal crossings.
Applying {\cite[Lemma~8.3]{Kar}}, the base change $X_{1}:=U_{1}\times_{H_{1}}B_{1}\to B_{1}$ is also weakly semistable, which is a weak semistable reduction of $f$.
Applying \cite{ALT}, we obtain a semistable reduction $f'\colon X'\to B'$ by taking a further base change $B'\to B_{1}$ and a modification $X'\to X_{1}\times_{B_{1}}B'$.

\end{document}